\documentclass[aap,MSNbibl,citesort,dvips]{arximspdf}
\usepackage{algorithm,algorithmic}
\usepackage{graphicx}

\doi{10.1214/11-AAP761}
\volume{22}
\issue{3}
\pubyear{2012}
\firstpage{931}
\lastpage{970}

\makeatletter
\newcommand{\R}{\mathbb{R}}
\newcommand{\Z}{\mathbb{Z}}
\newcommand{\N}{\mathbb{N}}
\newtheorem{thmm}{Theorem}
\newtheorem{lem}[thmm]{Lemma}
\newtheorem{prop}[thmm]{Proposition}
\newtheorem{cor}[thmm]{Corollary}
\newproclaim{dfn}[thmm]{Definition}
\newproclaim{remark}{Remark}
\newcommand{\ssi}{\mathcal{I}}
\newcommand{\ssj}{\mathcal{J}}
\newcommand{\pff}{\mathcal{P}}
\newcommand{\st}{\mathcal{T}}
\newcommand{\eqdist}{\stackrel{d}{=}}
\newcommand{\convdist}{\stackrel{d}{\to}}
\newcommand{\IPC}{\operatorname{IPC}}
\newcommand{\IIC}{\mathrm{IIC}}
\newcommand{\PGW}{\operatorname{PGW}}
\newcommand{\pgw}{\operatorname{PGW}}
\newcommand{\EDGE}{e}
\newcommand{\BOX}{B}
\newcommand{\HEIGHT}{h}
\newcommand{\LEFT}{\ell}
\newcommand{\BOXES}{\mathrm{BG}}
\makeatother

\begin{document}
\begin{frontmatter}

\title{Invasion percolation on the Poisson-weighted infinite tree}
\runtitle{Invasion percolation on the PWIT}

\begin{aug}
\author[A]{\fnms{Louigi} \snm{Addario-Berry}\corref{}\thanksref
{t1}\ead
[label=e1]{louigi@gmail.com}},
\author[B]{\fnms{Simon} \snm{Griffiths}\ead[label=e2]{sgriff@impa.br}}
\and
\author[C]{\fnms{Ross J.} \snm{Kang}\thanksref{t3}
\ead[label=e3]{ross.kang@gmail.com}}
\runauthor{L. Addario-Berry, S. Griffiths and R. J. Kang}

\thankstext{t1}{Supported by an NSERC Discovery Grant during this research.}
\thankstext{t3}{Supported by an NSERC Postdoctoral Fellowship during
this research.}

\affiliation{McGill University, IMPA and Durham University}

\address[A]{L. Addario-Berry\\
Department of Mathematics and Statistics\\
McGill University\\
1005-805 Sherbrooke West\\
Montreal, QC, K3A 2K6\\
Canada\\
\printead{e1}
}

\address[B]{S. Griffiths\\
IMPA\\
Estrada Dona Castorina 110\\
Rio de Janeiro, 22460-320, RJ\\
Brazil\\
\printead{e2}
}

\address[C]{R. J. Kang\\
School of Engineering \\
\quad and Computing Sciences\\
Durham University\\
South Road, Durham DH1 3LE\\
United Kingdom\\
\printead{e3}
}
\end{aug}

% HISTORY:
\received{\smonth{12} \syear{2009}}

% ABSTRACT
%
\begin{abstract}
We study invasion percolation on Aldous' Poisson-weighted infinite
tree, and derive two distinct Markovian representations of the
resulting process. One of these is
the $\sigma\to\infty$ limit of a representation discovered by Angel~et al.
[\textit{Ann. Appl. Probab.}
\textbf{36} (2008) 420--466]. We also
introduce an exploration process of a randomly weighted Poisson
incipient infinite cluster.
The dynamics of the new process are much more straightforward to
describe than those of invasion percolation, but it turns out that the
two processes have extremely similar behavior.
Finally, we introduce two new ``stationary'' representations of the
Poisson incipient infinite cluster as random graphs on $\Z$ which are,
in particular,
factors of a homogeneous Poisson point process on the upper half-plane
$\R\times[0,\infty)$.
\end{abstract}

% KEYWORDS
%
\begin{keyword}[class=AMS]
\kwd[Primary ]{60C05}
\kwd[; secondary ]{60G55}.
\end{keyword}

\begin{keyword}
\kwd{Invasion percolation}
\kwd{Prim's algorithm}
\kwd{Poisson-weighted infinite tree}
\kwd{percolation}
\kwd{random trees}.
\end{keyword}

\end{frontmatter}
%

%s1 ###
\section{Introduction}\label{secintro}
%!TEX root = unscaledprim.tex
Invasion percolation (or \textit{Prim's algorithm}~\cite{prim57shortest})
was first introduced by Jar{\'n}ik \cite
{jarnik30mst} as a procedure for
constructing the minimum weight spanning tree of a connected, weighted,
finite graph.
The procedure, however, may be applied to many infinite graphs without
modification. Given a connected graph $G=(V,E)$, a starting node $v_0 \in V$ and an
injective weight function $w\dvtx E \to\R$, the algorithm grows a component
from the root inductively, adding at each step the lowest weight edge
leaving the current component.
\begin{algorithm}[h!]
\begin{algorithmic}[1]
\item[]\textbf{For each }$i=0,1,\ldots:$
\item[]\textbf{1.} If $\{v_0,\ldots,v_i\}=V$, then stop.
\item[]\textbf{2.} Otherwise, let $e =uv \in E$ be the smallest weight
edge for
which\\
\item[]\quad$u \in\{v_0,\ldots,v_i\}$, $v \notin\{v_0,\ldots,v_i\}$.
\item[]\textbf{3.} Let $v_{i+1} = v$, and let $e_{i+1}=uv$.
\end{algorithmic}
\end{algorithm}

%
%%
%%\noindent\textbf{Prim}$(G,v_0,w)$
%for
%which $u \in\{v_0,\ldots,v_i\}$, $v \notin\{v_0,\ldots,v_i\}$. \\
%%

(Throughout the paper, the graphs and weight functions we consider will
be such that step 2, above, is well defined; i.e., the infimum of the
weights of all edges from $\{v_0,\ldots,v_i\}$ to the rest of the graph
is attained.)
If $|V|< \infty$, the resulting graph with vertex set $\{v_0,\ldots
,v_{|V|-1}\}$ and edge set $\{e_1,\ldots,e_{|V|-1}\}$ is the unique
minimum weight spanning tree of $G$. However, in general, for an
infinite graph, this procedure does not necessarily build a spanning
subgraph of $G$. In particular, if there is an infinite path leaving
$v_0$ and containing only edges of weight at most $h$, for some $h \in
\R$, then no vertex $v$ for which $\inf_{e \ni v} w(e) > h$ will ever
be explored.

Prim's algorithm was rediscovered under the name of invasion
percolation in the 1980s \cite
{lenormand1980description,chandler1982capillary}. The strong connection
between invasion percolation and critical percolation was immediately
recognized---a particularly nice example of this connection is
contained in the fact that invasion percolation on~$\Z^d$ occupies an
asymptotically zero proportion of the vertices of $\Z^d$ if and only if
the percolation probability at the critical point $p_c(\Z^d)$ is zero
(see Newman~\cite{newman1997topics}, page 24).

The well-known heuristic that percolation-style processes on $\Z^d$
should behave like percolation on a regular tree when $d$ is large led
Angel, Goodman, den Hollander and Slade~\cite{angel2006ipr} to study
invasion percolation on regular trees. Angel et al. prove far too many
results for us to summarize here. Among other topics, they study volume
growth and boundary growth, spectral and Hausdorff dimensions for the
set of vertices explored by invasion percolation. We hereafter refer to
this set---and to the subgraph induced by this set, which will cause
no confusion---as the \textit{invasion percolation cluster}. Their
results all stem from a Markovian representation of the invasion
percolation cluster as---informally---a single infinite path, at each
point of which is attached an independent random tree. (These trees are
``subcritical Bernoulli percolation clusters'' with a parameter which
becomes increasingly close to critical the further along the backbone
they are attached.) One of the major purposes of our paper is to
explore a new approach to this structural representation which applies
in some generality, so we take a moment to explain the representation
itself in more detail.

For the duration of the introduction, for integers $\sigma\geq2$, let
$\st_{\sigma}$ denote the infinite rooted $\sigma$-regular tree (each
node except the root has degree $\sigma+1$), with each edge $e$ labeled
by $U_e \sim\operatorname{Uniform}[0,\sigma]$ independently of all
other edges.
In general, for a weighted rooted graph $G$, let $G(p)$ be the
connected subgraph of~$G$ containing the root when all edges of weight
greater than $p$ are discarded. Let $p_1 = \inf\{p \dvtx\st_{\sigma
}(p) \mbox{ is infinite}\}$.
Then with probability one, $1 < p_1 < \sigma$, and $\st_{\sigma}(p_1)$
is infinite and contains precisely one edge $e_1$ of weight~$p_1$ (this
is not hard, and in particular follows from Corollary~\ref{corbackbone}
in Section~\ref{secrw}).
The component of $\st_{\sigma}(p_1)$ containing the root when $e_1$ is
removed is finite (or else we never would have explored edge~$e_1$).
Let $\st_{\sigma,1}$ be the component of $\st_{\sigma}(p_1)$
\textit{not}
containing the root when~$e_1$ is removed; then $\st_{\sigma,1}$, which
we view as rooted at its unique vertex which is an endpoint of $e_1$,
is infinite and contains only edges of weight less than~$p_1$.
Supposing we have defined $p_1,\ldots,p_i$, $e_1,\ldots,e_i$, and
$\st
_{\sigma,1},\ldots,\st_{\sigma,i}$, let $p_{i+1}=\inf\{p \dvtx
\st_{\sigma
,i}(p) \mbox{ is infinite}\}$. Then with probability one, $1 < p_{i+1}
< p_{i}$, and $\st_{\sigma,i}(p_{i+1})$ is infinite and contains
precisely one edge $e_{i+1}$ of weight $p_{i+1}$, which separates the
root of~$\st_{\sigma,i}$ from infinity. We define $\st_{\sigma
,i+1}$ to
be the component of $\st_{\sigma,i}(p_{i+1})$ not containing the root
when $e_{i+1}$ is removed, and root this tree at its unique vertex
which is an endpoint of $e_{i+1}$.

Now let $P=\{f_i\}_{i=1}^{\infty}$ be the unique path starting from the
root of $\st_{\sigma}$ and passing through all of $\{e_j\}
_{j=1}^{\infty
}$ (so $P$ is only a.s. defined). This path is called the \textit{
backbone} of the invasion percolation cluster.
The components of the invasion percolation cluster when all edges in $\{
e_{j}\}_{j=1}^{\infty}$ are removed are called \textit{ponds}; Angel
et al. also study the sizes of these ponds. There is further
interesting recent work on invasion percolation: on the sizes of ponds
for invasion percolation in $\Z^2$~\cite
{damron2008relations,damron2009outlets} and on rescaled invasion
percolation on trees~\cite{angel09rescaled}.
For integers $n \geq1$, let $W_n =W_n(\st_{\sigma}) = \sup_{j \geq n}
U_{f_j}$. Angel et al. term the process $\{W_n\}_{n=1}^{\infty}$ the
\textit{backbone forward maximal process} of $\st_{\sigma}$. $W_n$ is
nonincreasing and has $\lim_{n \to\infty} W_n=1$.
Note that $W_n > W_{n+1}$ only when~$f_n$ is one of the edges $e_j$, in
which case $W_n=U_{f_n}=p_j$.
Angel et al. prove that $\{W_n\}_{n=1}^{\infty}$ is a Markov process
and specify both its transition probabilities and its large-$n$
rescaled behavior.

The removal of the vertices and edges of $P$ separates the cluster into
components of finite size. Suppose $T$ is one such cluster and that its
neighbor on the path $P$ has distance $n$ from the root. Then Angel
et al. show that~$T$ is distributed as $\st_{\sigma}(W_n)$ conditioned
to stay finite, independently of all other components. This fact and
the results about the backbone forward maximal process mentioned in the
preceding paragraph form the heart of their structural results.

In this paper we introduce a new mechanism for studying invasion
percolation on randomly weighted trees, which can in particular give a
new perspective on the structural results of Angel et al. The
methodology works in some generality---in fact, parts of it are most
easily formulated as statements about invasion percolation on graphs
with deterministic weights. To apply such results, one then needs to
check that the hypotheses hold a.s. in a randomly weighted tree under
consideration (which in practice is always a trivial matter). We have
chosen to present our results in the setting where they are the most
simple and striking, which is that of the Poisson-weighted infinite
tree, or PWIT.

Informally, the PWIT can be described as follows. The root $r$ has a
countably infinite number of children $v_1,v_2,\ldots.$ The edges
$rv_1,rv_2,\ldots$ are assigned weights: for each $i \geq1$ the edge
$rv_i$ is weighted with the position of the $i$th point of a
homogeneous Poisson process of rate $1$ on $[0,\infty)$. [Equivalently,
starting from an infinite sequence of independent $\operatorname
{Exponential}(1)$ random variables $E_1,E_2,\ldots,$ for each $i$ the
edge $rv_i$ is given weight $E_1+\cdots+E_i$.]
This construction is repeated independently and recursively at each
child of the root. We may view the nodes of the PWIT as labeled by
$\bigcup_{i=0}^{\infty} \N^i$, so that the root has label
$\varnothing$
and in general, node $n_1n_2\cdots n_k$ has parent
$n_1n_2\cdots n_{k-1}$ and children $\{n_1n_2\cdots n_kn\}_{n \in\N}$;
however, this labeling will not play a major role in the paper.

The PWIT shows up as a standard large-$n$ limit for combinatorial
optimization problems on the complete graph $K_n$; see the excellent
survey paper by Aldous and Steele~\cite{aldous2004omp}
for details of how. Our case is
no exception; as one consequence of our study, we obtain novel proofs
of the main results of~\cite{mcdiarmid97mst}, about the early behavior
of Prim's algorithm on $K_n$ with i.i.d. uniform weights. Our main
results, however, link invasion percolation on the PWIT with the
Poisson \textit{incipient infinite cluster}---IIC, for short---constructed for general critical branching processes by Kesten \cite
{kesten1986subdiffusive}, but
earlier in the Poisson case by Grimmett
\cite{grimmett1980random}. The Poisson IIC is, informally, a~critical
Poisson Galton--Watson tree---$\pgw(1)$, for short---conditioned to
be infinite. There are at least two natural ways to formalize this
statement, but they both yield the same limiting construction, which we
now describe. Start with a~single, one-way infinite path, and then make
each node of the path the root of an independent copy of $\pgw(1)$. The
resulting infinite tree is the Poisson IIC, which we denote by $\st
_{\IIC}$.

For the remainder of the introduction, let $\st_0$ be a random weighted
tree with the distribution of the subgraph of the PWIT explored by
invasion percolation, with vertices $\{v_0,v_1,\ldots\}$ in order of
exploration, and let $\{W_i\}_{i=1}^{\infty}$ be its forward maximal
process. (We have not yet proved that $\st_0$ \textit{has} a~forward
maximal process, although the proof is straightforward---in
particular, this fact follows from Corollary~\ref{corbackbone} in
Section~\ref{secrw}.) Also, for any tree\vspace*{-2pt} $T$ and vertex $v$ of $T$,
let $T^{(v)}$ denote $T$ re-rooted at $v$. For two rooted random graphs
$G,H$, we write $G \eqdist H$ to mean $G$ and $H$ have the same
distribution in the local weak sense (i.e., neighborhoods of finite
order of the root have the same distribution in both graphs;\vspace*{-2pt} see \cite
{aldous2004omp}, Section~2, for more details). Similarly, we write
$G_n \convdist G$ to denote local weak convergence of a sequence $\{
G_n\}$ of rooted random graphs to a limiting random graph $G$. (This
notion of convergence in distribution deals only with the topological
structure of the graph, so in particular ignores any edge weights of
the graphs under consideration.)

Let $\pff$ be a homogeneous Poisson process of rate $1$ in the upper
half-plane $\R\times[0,\infty)$. Given two random variables $X$ and
$Y$, we say a random variable~$X$ is a \textit{factor} of $Y$ if almost
surely $X=f(Y)$ for some deterministic function~$f$. (Usage of this
term has not been fully standardized; ours agrees with that of \cite
{holroyd2009poisson}.)
The first main theorem of our paper is the following.
\begin{thmm}\label{main1}
There exist two $\pff$-a.s. distinct random trees $T=T(\pff)$,
$T'=T'(\pff
)$ with vertex set $\Z$ such that:
\begin{longlist}[(a)]
\item[(a)] in $T$ there is a unique infinite rightward path from each
vertex $\pff$-a.s.;
\item[(b)] in $T'$ there is a unique infinite leftward path from each
vertex $\pff$-a.s.;
\item[(c)] neither $T$ nor $T'$ is a factor of the other.
\end{longlist}
Furthermore, setting $U=T$ or $U=T'$, we have:
\begin{longlist}[(d)]
\item[(d)] for any $n \in\Z$, $U(\pff+n)=U(\pff)+n$;
\item[(e)] for any $n \in\Z$, $U^{(n)}$ is distributed as $\st
_{\IIC}$.
\end{longlist}
\end{thmm}

This theorem seems very similar in spirit to results of Ferrari, Landim
and Thorisson~\cite{ferrari2004poisson}, on tree and forest factors of
Poisson processes in $\R^d\times\R$, $d \geq1$ (with the final copy
of $\R$ viewed
as a time dimension). The graph they define is a tree when $d=1,2$ and
a forest when $d \geq3$. Some particular similarities of note: Ferrari
et al. explain how to use a preorder traversal (or \textit{depth-first
search}, a procedure quite similar to invasion percolation) of the
points of the Poisson process in order to view their trees as having
vertex set $\Z$; their graphs also have only one end (only one infinite
path leaving any vertex);
their graphs are built by joining each point to its first
time-successor within $\R^d$-distance one, yielding a ``coalescing
random walk'' interpretation of the construction, that is, reminiscent of
our random-walk description of the forward maximal process in Section
\ref{secrw}. Ferrari et al. do not explicitly identify the
distribution of the graph they define, but it would be very interesting
to know if it can be meaningfully interpreted as a higher-dimensional
analog of the Poisson IIC. Holroyd and Peres \cite
{holroyd2003trees} have also studied
tree and forest factors of Poisson point processes in $\R^d$, and
Holroyd and Peres~\cite{holroyd2003trees}, Tim{\'a}r \cite
{timar2004} have studied tree and forest factors of
general point process in $\R^d$.
Also, factors of one-dimensional Poisson processes that commute with
discrete shifts [i.e., as in Theorem~\ref{main1}(d), above] are one of
the subjects studied in~\cite{gurel09poisson}.

As a byproduct of the proof of Theorem~\ref{main1}, we will also obtain
the following theorem, which is a ``PWIT analog'' of \cite
{angel2006ipr}, Theorem 1.2.
\begin{thmm}\label{main2}
$\st_0^{(v_n)} \convdist\st_{\IIC}$ as $n \to\infty$.
\end{thmm}

Before stating our third theorem (in fact, the first two theorems lean
heavily on tools introduced in proving the third), we have a few more
concepts to introduce.
For each edge $e$ of $\st_{\IIC}$, let $X_e \sim\operatorname
{Uniform}[0,1)$, independently of all other edges. Let
$e_0=v_0v_1,e_1=v_1v_2,\ldots$ be the edges of the unique infinite path
(the backbone) in $\st_{\IIC}$, let $M_0=0$, and for integers $i \geq
1$, let $M_i = \max_{0 \leq j < i} X_{e_i}$. Now let $\st_{\IIC}^*$ be
the subtree of $\st_{\IIC}$ obtained as follows.
Let $v$ be a vertex of $\st_{\IIC}$, and let $v_i$ be the nearest
vertex of the backbone to $v$.\vadjust{\goodbreak} If any edge of the path from $v$ to
$v_i$ has weight greater than $M_i$, then remove $v$ from the tree. Do
this for each $v \in\st_{\IIC}$. Finally, remove $v_0$ and root at
$v_1$. The resulting subtree of $\st_{\IIC}$ is $\st^*_{\IIC}$.
\begin{thmm}\label{main3}
There is a continuous, strictly decreasing bijective map $q\dvtx
[1,\infty
)\to(0,1]$ such that $(q(W_1),q(W_2),\ldots) \eqdist(M_1,M_2,\ldots
)$, in the sense of finite-dimensional distributions.
Furthermore, $\st_0$ conditional on $(W_1,W_2,\ldots)$ is distributed
as $\st_{\IIC}^*$ conditional on $(M_1,M_2,\ldots
)=(q(W_1),q(W_2),\ldots
)$, in the local weak sense.
\end{thmm}

It is worth mentioning that the Markovian nature of $(W_1,W_2,\ldots)$
can be immediately deduced from this theorem. Given $M_i$, $M_{i+1}$ is
greater than~$M_i$ precisely if $X_{e_{i+1}} \in(M_i,1]$, in which
case $M_{i+1}=X_{e_{i+1}}$. Thus, given $M_i$, $M_{i+1}$ is equal to
$M_i$ with probability $(1-M_i)$, and otherwise is uniform on
$(M_i,1]$. Translating this to~$W_i$ immediately yields the ``PWIT
analog'' of the Markov process construction (\cite{angel2006ipr}, Proposition
3.1).

%s1.1 ###
\subsection{\texorpdfstring{The PWIT as a $\sigma\to\infty$ limit of $K_{\sigma+1}$ or of $\st_{\sigma}$}
{The PWIT as a sigma to infinity limit of K_{sigma+1}$ or of $T_{sigma}$}}

We mention in passing that with not much effort, it is possible to
prove convergence of invasion percolation on $\st_{\sigma}$ or on
$K_{\sigma+1}$ to invasion percolation on the PWIT, in a stronger sense
than the local weak sense.
Let ${\mathbf U} = (U_{1}^*,U_{2}^*,\ldots,U_{\sigma}^*)$ be the order
statistics of $\sigma$ independent $\operatorname{Uniform}[0,\sigma]$ random
variables.
Then~$U$ tends weakly to the vector of points of a homogeneous rate one
Poisson process $\pff$ on $[0,\infty)$.
More importantly for our current purpose, the vector
$(U_1,\ldots,U_{\lfloor\sqrt{\sigma}\rfloor})$ has total variation
distance $O(\sigma^{-1/2})$ from the vector of the first $\lfloor
\sqrt
{\sigma}\rfloor$ points of $\pff$.
It follows, in a sense that can easily be made precise, that \textit{the
first $o(\sqrt{\sigma})$ steps of invasion percolation on $\st
_{\sigma
}$} together have
total variation distance $o(1)$ from the same number of steps of
invasion percolation on the PWIT. A similar statement holds for the
first $o(\sqrt{\sigma})$ steps
of invasion percolation on $K_{\sigma+1}$. This in particular yields
new proofs of the explicit error bounds derived in \cite
{mcdiarmid97mst} for the behavior of the early stages of Prim's
algorithm on $K_{\sigma+1}$.
The details are straightforward, and we leave them to the interested reader.

%s1.2 ###
\subsection{Outline} In Section~\ref{secpwitpp} we construct the
building blocks on which the remainder of the article rests. In
particular, we describe a different way to view invasion percolation,
in terms of a ``note-taking'' procedure that accompanies the invasion
percolation procedure, and in the special case of invasion percolation
on trees contains all the information required to reconstruct the
original procedure. To best understand this note-taking procedure we
introduce the ``box process'' (Definition~\ref{dfnboxes}), which gives
us a clear picture of the connection mechanism of invasion percolation.
The box process also allows for an understanding of a related ``two-way
infinite'' invasion percolation process, which can be seen as
describing the behavior of invasion percolation far from the root.
Furthermore, with the introduction of the\vadjust{\goodbreak} ``box graph'' in Section \ref
{boxgraphs}, the box process itself becomes an interesting object of
study, and we derive some of its fundamental properties. Throughout
Section~\ref{secpwitpp}, our studies are in the deterministic setting.

In Section~\ref{secinvadepwit} we apply our tools to study $\st_0$. In
particular, we prove the PWIT analog of the forward maximal
representation of $\st_0$ in more detail. Section~\ref{secinvadepwit}
also contains some results concerning ballot style theorems, queueing
processes and Poisson Galton--Watson duality that are of use in proving
Theorems~\ref{main1}--\ref{main3}.

Finally, in Section~\ref{secstatproc} we prove a number of results
concerning the box graph and the stationary process. In particular, we
find that these graphs resemble the Poisson IIC locally. Using these
results, we deduce Theorems~\ref{main1}--\ref{main3}.

%s2 ###
\section{Redrawing invasion percolation}\label{secpwitpp}
In this section we describe a different way to view invasion
percolation which is at the heart of most of the results of this paper.
First, imagine keeping notes of the local edge landscape we see as we
perform invasion percolation, as follows.
At step $i$ of invasion percolation, we explore vertex $v_i$ and record
the weights of all edges leaving~$v_i$ and heading into new territory
by putting marks on the vertical half-line $\{i\}\times[0,\infty)$
whose heights are the weights of these edges.
(When performing invasion percolation on a rooted tree $T$, the edges
``heading into new territory'' are precisely the edges from $v_i$ to
its children in $T$.) Running the invasion percolation process until it
terminates (or forever) then yields some set
$P$ of points in the positive quadrant.

Formally, suppose $G=(V,E)$ is a weighted graph with all edge weights
distinct, and with distinguished vertex $v_0$. Then
the invasion percolation procedure defines an infinite subtree $T$ of
$G$, with vertex set $\{v_0,v_1,\ldots\}$.
For each $i \geq0$, let $p_i(1),p_i(2),\ldots,p_i(j_i)$ be the weights
of the edges from $v_i$ to $V\setminus\{v_0,\ldots,v_i\}$, in
increasing order of weight.\vspace*{-2pt} Let $P^i =P^i(G) = \{(i,p_i(j))\}_{j=1}^{j_i}$,
and let $P=P(G)=\bigcup_{i=0}^{|V|-1} P^i$.

In general, it is not possible to \textit{reconstruct} the steps taken by
invasion percolation by considering only the set $P$. However, this
\textit{is} possible for invasion percolation on trees, and we now
explain how.
In order to do so, we introduce an inductive procedure for building a
tree, given a set of points $P \subset\R^2$ and an interval $\ssi
\subset\Z$ of consecutive integers.
We write $n_{\ssi} =\inf\{n \in\ssi\} \geq-\infty$ and $m_{\ssi} =
\sup\{n \in\ssi\} \leq\infty$.

For notational convenience, given $X \subseteq\mathbb{R}^2$, we write
$|X|_P$ for $|P\cap X|$. Also, for a point $p \in\mathbb{R}^2$, we
write $x(p)$ for the $x$-coordinate and $y(p)$ for the $y$-coordinate.
Let us assume the following:
\begin{longlist}[1.]
\item[1.]\hypertarget{IPWIT1} All points of $P$ lie in the upper
half-plane. No
bounded set contains unboundedly many points.
\item[2.]\hypertarget{IPWIT2} For any $n \in\ssi$, there exists $k>0$
with $n-k
> n_{\ssi}-1$ for which $|[n-k,n)\times[0,\infty)|_P \ge k$.
\item[3.]\hypertarget{IPWIT3} $|P((-\infty,\infty)\times\{y\})| \le1$ for
any $y \in\mathbb{R}$.\vadjust{\goodbreak}
\end{longlist}
If $P$ satisfies these three conditions, we say it is \textit{reasonable}
(or \textit{$\ssi$-reasonable}, if $\ssi$ is not
clear from context). (Here, as well as later, we state deterministic
requirements for the point set $P$;
these requirements---and therefore, the results derived from them---will hold almost surely for all the random point sets we consider. In
particular, the reader will always be safe thinking of~$P$ as a Poisson
point set of intensity one in the upper half-plane.)
We start from an empty set $P_{n_{\ssi}}=\varnothing$, from which we will
build an increasing sequence of subsets of $\mathcal{N}_0$. The
following procedure requires $n_{\ssi} > -\infty$.

\begin{algorithm}
\begin{algorithmic}[1]
\item[]\textbf{For each }$i=n_{\ssi},n_{\ssi}+1,\ldots, m_{\ssi
}$:\\
\item[]\textbf{1.} Let $p_{i+1}=p_{i+1}(P,\ssi)$ be the point of
$([n_{\ssi},i+1)\times[0,\infty))\cap(P\setminus P_i)$
\item[] \quad minimizing
$y(p_{i+1})$.
\item[]\textbf{2.} Let $P_{i+1} =P_{i+1}(P,\ssi) = P_i \cup\{p_{i+1}\}
$, and let
$e_{i+1}=e_{i+1}(P,\ssi)=(i+1,$
\item[] \quad$\lfloor x(p_{i+1}) \rfloor)$.
\end{algorithmic}
\end{algorithm}

We refer to this procedure as \textit{point set invasion percolation}.
Since $P$ is reasonable, the procedure is well defined. The resulting
graph $\IPC(P,\ssi)$ has vertex set $\ssi$ and edge set $\{e_i
\dvtx n_{\ssi
} < i < m_{\ssi}+1\}$.
(We write $i < m_{\ssi}+1$ instead of $i \leq m_{\ssi}$ since we may
have $m_{\ssi}=\infty$, but $i$ is always finite.)
An example is shown in Figure~\ref{example}.
Note that $\IPC(P,\ssi)$ is a tree, which we view as rooted at
$n_{\ssi}$.
We often also view $\IPC(P,\ssi)$ as a weighted tree in which edge
$e_i$ has weight $y(p_i)$. In general in this section we work in the
deterministic setting. However, since our eventual aim is to link this
work to invasion percolation on randomly weighted trees we briefly
discuss how this can be done.

%
%%
%f1 ###
\begin{figure}[b]

\includegraphics{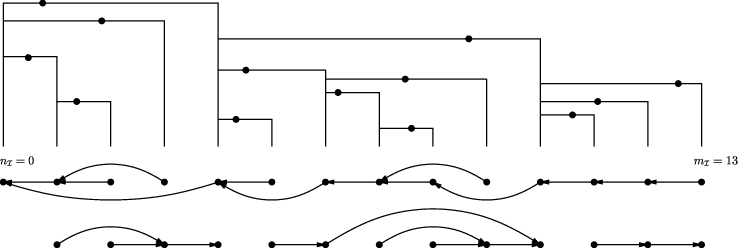}

\caption{Top, an $\ssi$-reasonable set of points $P$, with $\ssi=\{
0,\ldots,13\}$, and the corresponding boxes (defined in
Definition~\protect\ref{dfnboxes}). Middle, the tree $\IPC(P,\ssi)$.
Bottom, the forest $\BOXES(P,\ssi)$, defined at the start of Section
\protect\ref{boxgraphs}. All arrows point from child to parent.}
\label{example}
\end{figure}

%%
%%\input{example.pdf_t}
%%
%%
%%
%%
%{12}{14.4}{\rmdefault}{\mddefault}{\updefault}{\color
%[rgb]{0,0,0}$n_{\ssi}=0$}%
%}}}}
%{12}{14.4}{\rmdefault}{\mddefault}{\updefault}{\color
%[rgb]{0,0,0}$m_{\ssi}=13$}%
%}}}}
%%
%}
%
%%
%

\subsection*{IPC of the PWIT}
Now suppose that $\st$ is an instance of the PWIT, and let $\st_0$ be
the subtree of $\st$ explored by invasion percolation.
The following lemma is then immediate.

\begin{lem}
$\IPC(P(\st),\N)$ and $\st_0$ are identical, and for each $i \in\N$,
$w(e_i) = y(p_i)$.
\end{lem}

When performing invasion percolation on $\st$, for all $i$,
$P^i(\st)$ is a Poisson point process of rate $1$ on the vertical
half-line $\{i\} \times[0,\infty)$,
and $P(\st)$ is the union of these point processes.

We remark that since all points in $P$ have integer $x$-coordinates,
the floor in step~2, above, has no effect.
The use of the floor is to ensure that if a~point $p =(x,y)\in P$ is
replaced by a point $p'=(x',y)$, as long as $\lfloor x \rfloor=\lfloor
x' \rfloor$, the resulting
graph $\IPC(P,\ssi)$ will be unchanged. As a result we obtain the
following corollary.
\begin{cor}\label{ipcpoisson}
Let $\mathcal{P}$ be a Poisson point process of rate $1$ on $[0,\infty)
\times[0,\infty)$. Then $\IPC(\mathcal{P},\N)$ and $\st_0$ are
identically distributed.
\end{cor}
\begin{pf}
Associate to each point $p = (x(p),y(p))$ of $P(\st)$ an independent
uniform $U_p$, and let $p'$ be the point $(x(p)+U_p,y(p))$. Then $P' =
\{p' \dvtx p \in P\}$ is a Poisson point process of rate $1$ on
$[0,\infty)
\times[0,\infty)$, and
$\IPC(\mathcal{P}',\N)$ and $\IPC(P,\N)$ are identical. The result follows.
\end{pf}

This corollary reduces the study of the distributional properties of
$\st_0$ to that of the distributional properties of $\IPC(\mathcal
{P},\N
)$, where $\mathcal{P}$ is a Poisson point process of rate $1$ on
$[0,\infty) \times[0,\infty)$.

We also demonstrate how the two examples of invasion percolation
described in Section~\ref{secintro} can be encoded by suitable point processes.

\subsection*{IPC of an infinite randomly weighted $\sigma$-regular tree}
Let $\st_\sigma$ be the rooted regular tree with forward degree
$\sigma
\ge2$. We can model invasion percolation on~$\st_\sigma$ as follows:
for each $n \in\ssi= \N$, choose $\sigma$ independent, uniformly
random points of $[n,n+1)\times[0,\sigma)$ (or of $\{n\}\times
[0,\sigma)$).
Let $P$ be the union of all these points.

\subsection*{The minimum spanning tree of the complete graph}
Let $K_{\sigma+ 1}$ be the complete graph on $\sigma+ 1$ vertices. We
may approximately model invasion percolation on a randomly weighted
$K_{\sigma+ 1}$ as follows:
for each $n \in\ssi= \{0,\ldots,\sigma\}$, choose $\sigma- n$
independent, uniformly random points from the set $[n,n+1)\times
[0,\sigma)$.
Let $P$ be the union of all these points.

This representation is not exact due to the cycles in $K_n$. For example,
it is possible that the second least weight leaving the starting vertex
is on the edge between the second and third vertices visited by Prim's
algorithm.
However, the probability of events of this type is asymptotically
negligible for the first $o(\sqrt{\sigma})$ steps of the algorithm.

The acyclicity of trees is what allows us to model them by a point
process without reference to the order of exploration of vertices. In
general---for invasion percolation on $\Z^d$, for example---it may
still be possible
to use some\vadjust{\goodbreak} of the following methodology while jointly constructing the
point process $P$ and the exploration process ``as we go.'' However, we
have not pursued this avenue of study.

For the remainder of the section, we explore what properties we can
derive about the point process invasion percolation procedure with as
few restrictions on the point set $P$ as possible.
The next definitions and lemma provide an alternative geometric
characterization of the connection rule used in the above inductive
procedure, one that will be useful throughout the paper.
\begin{dfn}\label{dfnboxes}
Given an interval $\ssi$, with $n_{\ssi} > -\infty$, and an $\ssi
$-reasonable point set $P$, for each $i \in\ssi$ with $i > n_{\ssi
}$, let
\[
\HEIGHT_i(P,\ssi) = \inf\{h \dvtx\exists j\in\ssi, n_{\ssi}
\leq j <
i \mbox{ such that } |[j,i)\times[0,h]|_P \geq i-j\}.
\]
Let $\LEFT_i(P,\ssi)$ be the minimum integer $\LEFT_i \in[n_\ssi,i)$
such that
$|[\LEFT_i,i)\times[0,\HEIGHT_{P,\ssi}(i)]|_P = i-\LEFT_I$,
let $\BOX_i(P,\ssi)= [\LEFT_i,i)\times[0,\HEIGHT_i]$
and let $t_i(P,\ssi)$ be the unique point in $\BOX_i$ with
$y(t_i)=\HEIGHT_i$.
\end{dfn}

We often omit reference to the parameters $P$ and $\ssi$ if the context
is clear.

We take a moment to observe that these functions are well defined. It
follows from condition \hyperlink{IPWIT2}{2} that $h_{i}$ is finite,
and from
condition \hyperlink{IPWIT1}{1} that it is positive. The minimality of $h_{i}$
then implies the existence of a point $p \in P(\BOX_i)$ such that $y(p)
= \HEIGHT_i$. The fact there is a unique such point follows from
condition \hyperlink{IPWIT3}{3}.
\begin{lem}\label{boxesproc}
If $n_{\ssi} > -\infty$ and $P$ is $\ssi$-reasonable, then for all $n
\in\ssi\setminus\{n_{\ssi}\}$, we have $t_n=p_n$.
\end{lem}
\begin{pf}
It suffices to show (by condition \hyperlink{IPWIT3}{3}) that $y(t_n) = y(p_n)$.
We prove this by induction on $n$. Clearly, the assertion holds for $n
= n_\ssi+ 1$. Assume $n > n_\ssi+ 1$ and that $t_i = p_i$ for all
$n_\ssi+ 1 \le i < n$. First, since $|\BOX_n|_P = n-\LEFT_n$ and
$\bigcup_{i = n_\ssi+ 1}^{n-1}p_i$ contains at most $n - \LEFT_n - 1$
points of $P \cap B_n$, the set $(P\setminus P_{n-1})\cap\BOX_n$
contains at least one point
and so $y(p_n) \le y(t_n)$.

To show that $y(t_n) \le y(p_n)$,
first note that if $x(p_n)\geq n-1$, then $|[n-1,\break n)\times[0,y(p_n)]|_P
\geq1$ and so certainly $y(t_n) \leq y(p_n)$.
We thus assume that $x(p_n)<n-1$ and construct a sequence $\{a_i\}_{i =
0}^k$ inductively as follows:

\begin{algorithm}
\begin{algorithmic}[1]
\item[] Let $i = 0$ and let $a_0 = n - 1$.
\item[]\textbf{1.} If $a_i \le x(p_n)$, set $k = i$ and \textbf
{stop}.
\item[]\textbf{2.} Otherwise, let $a_{i+1}=\LEFT_{a_i}$, then let $i = i
+ 1$ and return to \textbf{1}.
\end{algorithmic}
\end{algorithm}

For each $0 \leq i \leq k$ for which $a_i$ is defined, if $a_i >
x(p_n)$, then $y(p_{a_i}) < y(p_n)$ or else the point $p_n$ was a
better choice for $p_{a_i}$. By the inductive hypothesis, $p_{a_i} =
t_{a_i}$. By construction,\vspace*{-1pt} $|\BOX_{a_i}|_P = a_i - \LEFT_{a_i} =
a_i-a_{i+1}$ for all $i < k$.
Since $\BOX_{a_i} \cap\BOX_{a_j} = \varnothing$ for all\vadjust{\goodbreak} $i \ne j$, we
conclude that $\bigcup_{i=0}^{k-1} \BOX_{a_i}$ has $n - 1 - a_k$ points of~$P$.
Thus, $|[a_k,n)\times[0,y(p_n)]|_P \ge|\{p_n\}\cup\bigcup
_{i=0}^{k-1} B_{a_i}|_P \ge n - a_k$. By the choice of $\HEIGHT_{n}$
minimum, it follows that $\HEIGHT_{n}=y(t_n) \le y(p_n)$ as required.
\end{pf}

The structure of the containment relations among the boxes $B_i$ turns
out to be interesting in its own right, and we explore aspects of it
here as well as later in the paper.

\begin{lem}\label{lemleftboxtaller}
If $n_{\ssi} > -\infty$ and $P$ is $\ssi$-reasonable, then for $n
\in
\ssi\setminus\{n_{\ssi}\}$, either $\HEIGHT_{\LEFT_n} >\HEIGHT_n$ or
$\LEFT_n = n_{\ssi}$.
\end{lem}
\begin{pf}
Assume $\LEFT_n \neq n_{\ssi}$,
suppose $\HEIGHT_{\LEFT_n} \le\HEIGHT_{n}$ and write $m=\LEFT_n$.
Then both~$\BOX_{m}$ and $\BOX_{n}$ are contained in $[\LEFT
_m,n)\times
[0,\HEIGHT_n]$,
so
$|[\LEFT_{m},n)\times[0,\HEIGHT_n]|_P \ge n - \LEFT_m$.
This contradicts either the choice of $\HEIGHT_n$ or the choice of
$\LEFT_n$.
\end{pf}

\begin{lem}\label{lemboxcontainment}
If $n_{\ssi} > -\infty$ and $P$ is $\ssi$-reasonable, then for any $i,j
\in\ssi\setminus\{n_{\ssi}\}$ with $i < j$,
either $\BOX_i \cap\BOX_j = \varnothing$ or $\BOX_i \subseteq\BOX_j$.
\end{lem}
\begin{pf}
Suppose that $B_{i}\cap B_{j}\neq\phi$. In particular this implies
$\LEFT_{j}<i$. We prove that $B_{i}\subseteq B_{j}$ by proving that
$h_{i}<h_{j}$ and $\LEFT_{j}\le\LEFT_{i}$.

The minimality of $h_{j}$ implies that $|[\LEFT_{j},j)\times
[0,h_{j})|_{P}=j-\LEFT_{j}-1$ and that $|[i,j)\times[0,h_{j})|_{P}\le
j-i-1$. Thus $|[\LEFT_{j},i)\times[0,h_{j})|\ge i-\LEFT_{j}$, which
immediately implies that $h_{i}<h_{j}$.

We now prove $\LEFT_{j}\le\LEFT_{i}$. Suppose that $\LEFT_{i}<\LEFT
_{j}$. Then, by reasoning as above, and using the fact that
$h_{j}>h_{i}$ we have that $|[\LEFT_{i},\LEFT_{j})\times
[0,h_{j})|_{P}\ge|[\LEFT_{i},\LEFT_{j})\times[0,h_{i})|_{P} \ge
\LEFT
_{j}-\LEFT_{i}$. This implies that $|[\LEFT_{i},j)\times
[0,h_{j})|_{P}\ge j-\LEFT_{i}$, which contradicts the definition of
$\LEFT_{j}$.
\end{pf}

\begin{lem}\label{lemboxpath}
If $n_{\ssi} > -\infty$ and $P$ is $\ssi$-reasonable, then for any $i,j
\in\ssi\setminus\{n_{\mathcal I}\}$
such that $\LEFT_j < i \le j$, there is a path in $\IPC(P,\ssi)$
between $i$ and $\LEFT_j$.
\end{lem}

\begin{pf}
Observe that since $\LEFT_j < i$, $B_i \subset B_j$ by Lemma \ref
{lemboxcontainment}.
We apply induction on $i-\LEFT_j$. If $i-\LEFT_j=1$, then we must have
$\lfloor x(p_i)\rfloor= \LEFT_j$,
so $\EDGE_i = (i,\LEFT_j)$, verifying the claim.

For larger values of $i-\LEFT_j$, first note that since $B_i \subset
B_j$, we must have
$\LEFT_j \leq\LEFT_i\leq\lfloor x(p_i)\rfloor< i$. If $\LEFT_j
=\lfloor x(p_i)\rfloor$, then $\EDGE_i$ is a path from
$i$ to $\LEFT_j$. Otherwise, $\lfloor x(p_i)\rfloor-\LEFT_j < i -
\LEFT
_j$, so by induction there is a path from
$\lfloor x(p_i)\rfloor$ to $\LEFT_j$, which together with edge $\EDGE
_i$ yields a path from $i$ to $\LEFT_j$.
\end{pf}

%s2.1 ###
\subsection{Point process invasion percolation in the upper half-plane}
For suitable point sets $P$, we may hope to define a version of the
invasion percolation procedure
in which $\ssi=\Z$ (or more generally when $n_{\ssi}=-\infty$).
This is
indeed possible, and the
resulting infinite graph can be said to capture the behavior of invasion
percolation ``very far from the root.'' A direct inductive description of
the graph seems difficult, and so we define the object $\IPC(P,\Z)$
as the limit of $\IPC(P,\Z\cap[m,\infty))$ as\vadjust{\goodbreak} $m \to-\infty$.
Later, we shall also see how the alternative characterization
of the connection rule given by Definition~\ref{dfnboxes} and Lemma
\ref{boxesproc}
can be used to define this extension of the invasion percolation procedure.

As before, we desire as few restrictions on $P$ as possible. In this
section, we suppose we are given a set of points $P \subset\R^2$ and
an interval $\ssi$ with $n_{\ssi} \geq-\infty$ and $m_{\ssi} \leq
\infty$.
We say that $P$ is \textit{seemly} (or \textit{$\ssi$-seemly}, if $\ssi
$ is
not clear from context) if $P$ satisfies conditions \hyperlink{IPWIT1}{1}--\hyperlink{IPWIT3}{3}
and additionally either (a) $n_{\ssi
}>-\infty$,
or (b)~$n_{\ssi}=-\infty$ and $P$ satisfies conditions \hyperlink{SIPWIT1}{4}
and \hyperlink{SIPWIT2}{5}, below.
\begin{longlist}[4.]
\item[4.]\hypertarget{SIPWIT1} For any $n \in\ssi$, there are
infinitely many $m
\in\ssi\cap(-\infty,n)$ such that $|[m,n)\times[0,1]|_P > n - m$.
\item[5.]\hypertarget{SIPWIT2} If $\lambda< 1$, then for any $n \in\ssi
$ there
are at most finitely many $m \in\ssi\cap(-\infty,n)$ such that
$|[m,n)\times[0,\lambda]|_P \ge n - m$.
\end{longlist}
The reader can verify that the following two examples almost surely
produce seemly point sets.

\subsection*{Stationary limit of IPC on $\st_\sigma$}
Let $P$ be defined by choosing $\sigma$ independent, uniformly random
points in the set $[n,n+1)\times[0,\sigma)$ for each $n \in\ssi= \Z$.

\subsection*{Stationary limit of the Poisson IPC}
Let $P$ be a Poisson point process of intensity~$1$ in the upper half
plane, and let $\ssi= \Z$.

The following lemma essentially states that for $\ssi$-seemly point
sets with $n_{\ssi}=-\infty$, all edges have weight less than $1$.
\begin{lem}\label{lemweightone}
If $n_{\mathcal I}=-\infty$ and $P$ is $\ssi$-seemly, then for any
$n\in\ssi$ there exists $m_0 \in\ssi$ such that $\HEIGHT_n(P,\ssi
\cap
[m,\infty)) < 1$ for all integers $m \leq m_0$.
\end{lem}
\begin{pf}
By condition \hyperlink{SIPWIT1}{4}, $|[m,n)\times[0,1]|_P > n - m$ for
infinitely many integers $m < n$; therefore, $\HEIGHT_n(P,\ssi\cap
[m,\infty)) < 1$ for infinitely many integers $m < n$.
But $\HEIGHT_n(P,\ssi\cap[m,\infty)) = y(p_n(P,\ssi\cap[m,\infty)))$,
and $y(p_n(P,\ssi\cap[m,\infty)))$ is nonincreasing as $m$ decreases,
so by condition \hyperlink{IPWIT3}{3} $y(p_n(P,\ssi\cap[m,\infty))) <
1$ for all~$m$ small enough.
\end{pf}

We next consider the family of intervals $\ssi\cap[m,\infty)$ for $m
\in\Z$, and show that as $m \to-\infty$, each vertex only changes its
parent a finite number of times.
This allows us to consistently define the limiting object $\IPC(P,\ssi)$.
\begin{lem}\label{lemconsistent}
If $P$ is $\ssi$-seemly, then for any $n \in\ssi$, there exists $m_0 >
-\infty$ such that
$p_n(P,\ssi\cap[m,\infty)) =p_n(P,\ssi\cap[m_0,\infty))$ for all
$m \in
\ssi\cap(-\infty, m_0]$.
\end{lem}
\begin{pf}
The lemma is obvious if $n_{\ssi}>-\infty$ so assume $n_{\ssi
}=-\infty$.
Fix $n \in\ssi$ and suppose the assertion of the lemma fails for this $n$.
Then there exists a strictly decreasing integer sequence $\{m_i\}
_{i=0}^{\infty}$
and a sequence $\{q_i\}_{i=0}^{\infty}$ of distinct points in~$P$ such
that $p_n(P,\ssi\cap[m_i,\infty))=q_i$
for all $i \in\N$, whose $y$-coordinates decrease\vadjust{\goodbreak} strictly as $i$
increases. By Lemma~\ref{lemweightone},
there exists some $i_0$ such that $y(q_i) < 1$ for all $i \geq i_0$.
But then for all $i \geq i_0$,
$\BOX_n(P,\ssi\cap[m_i,\infty)) \subset[\LEFT_n(P,\ssi\cap
[m_i,\infty
)),n)\times[0,y(q_i)]$,
and so for such $i$,
\begin{eqnarray*}
\bigl|\bigl[\LEFT_n\bigl(P,\ssi\cap[m_i,\infty)\bigr),n\bigr)\times
[0,y(q_i)]\bigr|_P &\geq&\bigl|\BOX
_n\bigl(P,\ssi\cap[m_i,\infty)\bigr)\bigr|_P \\
&\geq& n-\LEFT_n\bigl(P,Z\cap[m_i,\infty)\bigr).
\end{eqnarray*}
This is a contradiction to condition \hyperlink{SIPWIT2}{5}.
\end{pf}

For a seemly point set $P$, we now define $\IPC(P,\ssi)$ to be the
graph with vertex set $\ssi$ and such that for each $n \in\ssi$, $e_n
= e_n(P,\ssi) = \lim_{m \to-\infty} e_n(P,\ssi\cap[m,\infty))$. This
limit is well defined by the preceding lemma.
We likewise define $p_n(P,\ssi)$, $\LEFT_n(P,\ssi)$, $\HEIGHT
_n(P,\ssi
)$ and $\BOX_n(P,\ssi)$.
By a limiting argument
Lemmas~\ref{lemleftboxtaller},~\ref{lemboxcontainment} and \ref
{lemboxpath} are also valid with respect to $\IPC(P,\ssi)$ when
$n_{\ssi}=-\infty$. We therefore obtain the following theorem.
\begin{thmm}\label{thmtree}
If $P$ is $\ssi$-seemly, then $\IPC(P,\ssi)$ is a tree.
\end{thmm}
\begin{pf}
Since it is clearly acyclic, we just need to show that $\IPC(P,\ssi)$
is connected. Suppose $i,j \in\ssi$, $i < j$. Let $\LEFT_j^0=\LEFT_j$
and for $t \geq1$, $t \in\N$, let $\LEFT^t_j = \LEFT_{\LEFT^{t-1}_j}$.
Then there must exist $t \in\N$ such that $i \in\BOX_{\LEFT^t_j}$.
By Lemma~\ref{lemboxpath},\vspace*{-3pt} there is a path between $j$ and $\LEFT
^{t+1}_j$, and there is a path between $i$ and $\LEFT^{t+1}_j$. As $i$
and $j$ were arbitrary, this completes the proof.
\end{pf}

An advantage of the current formulation of invasion percolation is that
we can equivalently define
the limit process via conditions on the numbers of points in boxes
$[m,n)\times[0,h]$.
More precisely, the following lemma is easily verified.

\begin{lem}\label{lemequivalent}
Suppose $P$ is $\ssi$-reasonable.
Fix $k,n$ with $n_{\ssi} < k < m_{\ssi}+1$ and $0 < n < (k-n_{\ssi
})+1$, and $y > 0$.
In order that $\LEFT_k=k-n$ and that $\HEIGHT_0=y$, it is necessary and
sufficient that the following three conditions hold:
\begin{itemize}
\item$|[k-n,k)\times[0,y]|_{\pff}=n$ and $|[k-n,k)\times
[0,y)|_{\pff
}=n-1$ [call this condition $E=E(k-n,k,y,P)$].
\item For all $0 < m \leq n$, $|[k-m,k)\times[0,y)|_{\pff}<m$ [call this
condition $F=F(k-n,k,y,P)$].
\item For all $m \in\N$, $|[k-n-m,k-n)\times[0,y]|_{\pff}< m$ [call
this condition $G=G(k-n,y,P)$].
\end{itemize}
In this case, $\BOX_k=[k-n,k]\times[0,y]$, $p_n$ is the unique point $p
\in P$ with $y(p)=y$, and $\EDGE_n=(n,\lfloor x(p_n) \rfloor)$.
\end{lem}

We will sometimes have use for the condition $G(k-n,y^-)$, which is the
same as the condition $G$ above but with $[0,y]$ replaced by $[0,y)$.
The next lemma provides a condition under which we can determine the
behavior to the right of a given integer $n$ without
further reference to the behavior of $P$ to the left of $n$. Its proof
is obvious and is omitted.\vadjust{\goodbreak}
\begin{lem}\label{lemrestrict}
Suppose $P$ is $\ssi$-reasonable.
Fix $n_{\ssi} < n < m_{\ssi}+1$ and $y > 0$,
let $Q=\{p \in P \dvtx x(p) \geq n, y(p) \leq y\}$ and let $\ssj=\ssi
\cap
\{n,\ldots,\infty\}$.
If $G(n,y^-)$ holds and $Q$ is $\ssj$-reasonable, then
for all $m$ with $n < m < m_{\ssi}+1$, $\LEFT_m(P,\ssi)=\LEFT
_m(Q,\ssj
)$, $\HEIGHT_m(P,\ssi)=\HEIGHT_m(Q,\ssj)$
and $p_n(P,\ssi)=p_n(Q,\ssj)$.
\end{lem}

We will also have use of the following sufficient condition for $Q$ to
be reasonable. (Again, the proof is straightforward and is omitted.)
\begin{lem}\label{reasonablecond}
Let $\ssi,P,k,n$ and $y$ be as in Lemma~\ref{lemequivalent},
let $\ssj=\{k-n,\ldots,k\}$ and let $Q=P \cap([k-n,k]\times[0,y])$.
If $E,F$ and $G$ all hold, then $Q=P \cap\BOX_k$ and $Q$ is $\ssj
$-reasonable.
\end{lem}
%
%s2.2 ###
\subsection{Box graphs}\label{boxgraphs}

As we saw above, the boxes $B_n$ play a useful role in our study of
invasion percolation.
The boxes can also be seen to capture information about the structure
of the point process invasion
percolation procedure itself. For example, it is easily checked that if
the procedure explores some edge $e$ lying within a box $B_n$, then it
will explore all other edges lying within $B_n$ before exploring any
edges with an endpoint outside of $B_n$. (Of course, the procedural
interpretation does not exist when $\ssi=\Z$, but
in this case we can still think of the boxes as capturing information
about the process behavior ``far from the root.'')

In this section, we introduce a graph which characterizes the
containment relation among the boxes. Given an $\ssi$-reasonable point
set $P$, we define $\BOXES(P,\ssi)$ to be the graph with vertex set
$\ssi\setminus\{n_\ssi,+\infty\}$ and such that, for $i < j$, $i$ and
$j$ are joined by an edge if and only if $\BOX_i(P,\ssi) \subseteq
\BOX
_j(P,\ssi)$ and $\BOX_i(P,\ssi) \not\subseteq\BOX_{j'}(P,\ssi)$ for
any $i < j' < j$.
Also, for $i \in\ssi\setminus\{n_{\ssi},m_{\ssi}\}$, we write
$a_i(P,\ssi)$ for the parent of~$i$ in $\BOXES(P,\ssi)$.

%
%
%f2 ###
\begin{figure}[b]

\includegraphics{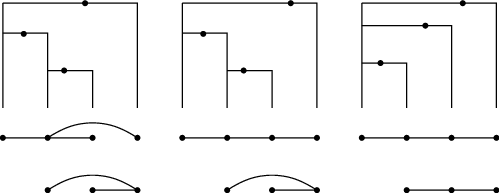}

\caption{The leftmost and middle sets of points have the same $\BOXES$
graphs but different $\IPC$ graphs. The middle and rightmost sets of
points have the same $\IPC$ graph but different $\BOXES$ graphs.}
\label{example2}
\end{figure}

The examples shown in Figure~\ref{example2} demonstrate that between
the graphs $\IPC(P,\ssi)$ and $\BOXES(P,\ssi)$, neither is determined
by the other. [Theorem~\ref{main1}(c) is essentially a consequence of
this fact.]

Clearly, $\BOXES(P,\ssi)$ is acyclic for any $\ssi$. We shall show that
$\BOXES(P,\Z)$ is a tree (i.e., connected) under the additional
assumption of the ``rightward version'' of condition \hyperlink{SIPWIT2}{5}.

\begin{longlist}[6.]
\item[6.]\hypertarget{SIPWIT4} If $\lambda< 1$, then for any $m \in\ssi
$, there
are at most finitely many $n \in\Z\cap(m,\infty)$ such that
$|[m,n)\times[0,\lambda]|_P \ge n - m$.
\end{longlist}

If $P$ satisfies conditions \hyperlink{IPWIT1}{1}--\hyperlink
{SIPWIT4}{6} with
$\ssi= \Z$, we say that $P$ is \textit{exemplary}. Both examples of the
last subsection are almost surely exemplary point sets.

\begin{lem}\label{lemboxgraph2}
Suppose $m_{\ssi} = \infty$ and $P$ is an $\ssi$-reasonable point set
that satisfies condition \hyperlink{SIPWIT4}{6}. Choose any $m \in\ssi
\setminus
\{n_{\ssi},+\infty\}$ for which $\HEIGHT_m < 1$ and for which there is
no $m' \in\Z\cap(m,\infty)$ such that $\BOX_m \subseteq\BOX_m'$ and
$\HEIGHT_m' \ge1$. Then there are infinitely many $n \in\Z\cap
(m,\infty)$ such that $\BOX_m \subseteq\BOX_n$.
\end{lem}
\begin{pf}
Suppose $m$ is as in the statement of the lemma but that there are only
finitely many
$n > m$ such that $\BOX_m \subseteq\BOX_n$. Then by replacing $\BOX_m$
with the tallest box that
contains it, we may assume that in fact there is no $n > m$ such that
$\BOX_m \subseteq\BOX_n$.
By condition \hyperlink{SIPWIT4}{6}, we may choose $n> m$ for which
$|[m,n)\times
[0,\HEIGHT_m]|_P < n-m$.
Thus, there must be $i \in\{m+1,\ldots,n\}$ for which $\HEIGHT_i >
\HEIGHT_m$, so take $i$ minimum such that
this holds. By Lemma~\ref{lemleftboxtaller}, we must then have $\LEFT
_i < m$, and so by Lemma~\ref{lemboxcontainment} we must have $\BOX_m
\subset\BOX_i$, a~contradiction.
\end{pf}

Before showing that $\BOXES(P,\Z)$ is a tree, let us first use the
lemma to confirm the basic property of exemplary point sets that every
point of $P$ under the line $y = 1$ lies along the top of some box
$\BOX_n$.

\begin{prop}\label{propunder1}
If $P$ is exemplary, then for all $p \in P\cap[n_\ssi,+\infty)\times
[0,1)$, we have $p = p_n$ for some $n \in\ssi$.
\end{prop}

\begin{pf}
Let $p \in P$ have $y(p) < 1$. We first note that if $p \in\BOX_m$ for
some~$m$, then $p = p_n$ for some $\lceil x(p) \rceil\leq n \leq m$.
Also, there must be some integer $k \leq x(p)$ for which $\HEIGHT_k >
y(p)$, or else
$|[\lceil x(p)\rceil-i,\lceil x(p)\rceil]\times[0,y(p)]|_P \geq i$ for
infinitely many integers $i > 0$, which contradicts condition \hyperlink
{SIPWIT2}{5}.

By Lemma~\ref{lemweightone}, $\HEIGHT_m < 1$ for all $m \in\Z$, so by
Lemma~\ref{lemboxgraph2}, there are infinitely many
$m \in\Z$ for which $\BOX_k \subseteq\BOX_m$. One of these boxes
contains $p$, so $p=p_n$ for some $n$, as claimed.
\end{pf}

\begin{thmm}\label{thmbgconnect}
If $P$ is exemplary, then $\BOXES(P,\Z)$ is a tree.
\end{thmm}
\begin{pf}
It suffices to show that $\BOXES(P,\Z)$ is connected.
Recall that \mbox{$\HEIGHT_n < 1$} for any $n \in\Z$, by Lemma \ref
{lemweightone}. Fix $i < j$, $i,j \in\Z$.
By Lemma~\ref{lemboxgraph2} there are infinitely many $m$ such that
$\BOX_i \subset\BOX_m$.
Take the least such $m$ for which $m \geq i$---then also $\BOX_j
\subseteq\BOX_m$, and so by Lemma~\ref{lemboxpath}
there exist paths from $i$ to~$m$ and from $j$ to $m$. The theorem follows.
\end{pf}

%s2.3 ###
\subsection{Random walks and the forward maximal process}\label{secrw}

Let $P$ be a point set satisfying condition \hyperlink{IPWIT1}{1}.
Given $h > 0$ and $k \in\ssi$, we define random walks
$S^{k,h}=S^{k,h}(P)$ and $L^{k,h}=L^{k,h}(P)$ as follows. We set
$S^{k,h}_0=L^{k,h}_0 = 0$ and,
for $i \geq1$, set $S^{k,h}_i = | [k,k+i]\times[0,h] |_{P}-i$, and set
$L^{k,h}_i = | [k,k-i]\times[0,h] |_{P}-i$.
We also define random walks $S^{k,h^-}$ and $L^{k,h^-}$, by replacing
$[0,h]$ by $[0,h)$ in the above definitions. In other words,
the random walks $S^{k,h^-}$ and $L^{k,h^-}$ ignore points on the line
$y=h$. (For fixed $h$, for any of the random point sets~$P$ we will
consider, it will be the case that with probability $1$, $S^{k,h}_i=S^{k,h^-}_i$
for all $i$, but we will at times work in conditional settings in which
these two random walks are not identical.)

We say that $S^{k,h}$ \textit{survives} if for all $i \geq0$, $S^{k,h}_i
\geq0$, and otherwise say that~$S^{k,h}$ \textit{dies}. Also, we say that
$S^{k,h}$ \textit{has a
chance} if $S^{k,h}_i \geq0$ for some $i > 0$, and otherwise that $S^{k,h}$ \textit{
has no chance}. We extend these definitions to $L^{k,h}$ by symmetry.

We now establish two more basic properties of $\IPC(\mathcal{P},\ssi)$,
under the following additional assumptions.

\begin{longlist}[7.]
\item[7.]\hypertarget{SIPWIT3} If $\lambda> 1$, then for any $m \in\Z$,
$S_n^{m,\lambda} \leq n$ for at most finitely many $n \in\N$.
%%$|[m,n)
%
\item[8.]\hypertarget{RWdies} $S^{k,1}$ dies for all $k \in\ssi$.
\end{longlist}
Roughly speaking, condition \hyperlink{SIPWIT3}{7} is a ``rightward
version'' of
condition \hyperlink{SIPWIT1}{4}. If~$P$ is an $\ssi$-reasonable point
set that
satisfies conditions \hyperlink{SIPWIT4}{6}, \hyperlink{SIPWIT3}{7} and
\hyperlink{RWdies}{8}, we
say that~$P$ is \textit{distinguished} (or \textit{$\ssi
$-distinguished}, if
$\ssi$ is not
clear from context). The first two examples given in the introduction
to this section are almost surely distinguished point sets.

We will see that for distinguished point sets $P$, when $n_{\mathcal I}
> -\infty$ and $m_{\mathcal I}=\infty$,
$\BOXES(P,\mathcal I)$ is \textit{not} connected---in this case we call
the connected components the \textit{ponds}
of $\BOXES(P,\mathcal I)$. We will see later that this agrees with the
normal use of this term in the invasion percolation literature.

\begin{lem}\label{lemboxgraph1}
If $n_{\ssi} > -\infty$, $m_{\ssi} = \infty$ and $P$ is $\ssi
$-distinguished,
then for any $m \in\ssi\setminus\{n_{\mathcal I},+\infty\}$, if
$\HEIGHT_m \ge1$,
then there are at most finitely many $n \in\Z$, $n > m$, such that
$\BOX
_m \subseteq\BOX_n$.
\end{lem}
\begin{pf}
Suppose otherwise. Without loss of generality, we may assume that
$\HEIGHT_m > 1$. Consider the integer sequence $\{n_i\}_{i=0}^\infty$,
which is defined as follows. Let $n_0 = m$. For $i \in\N$, let
$n_{i+1}$ be the smallest integer greater than $n_i$ such that $\BOX
_{n_i} \subseteq\BOX_{n_{i+1}}$.
Then for any $i \in\N$ and all $n_i < n < n_{i+1}$, we have
$\BOX_n \subseteq[n_i,n_{i+1})\times[0,\HEIGHT_{n_i}]$ for all $n_i <
n < n_{i+1}$ by Lemmas~\ref{lemleftboxtaller} and \ref
{lemboxcontainment}. Furthermore, it follows from the definition of
$\BOX_n$ and Lemma~\ref{lemboxcontainment} that
$|([n_i,n_{i+1})\times
[0,\HEIGHT_{n_{i+1}})) \setminus(\bigcup_{n_i < n < n_{i+1}} \BOX
_n)|_P =
0$ (or otherwise there would be a smaller choice for $\HEIGHT_{n_{i+1}}$).
Thus, $|[n_i,n_{i+1})\times[0,\HEIGHT_{n_i}]|_P = |\bigcup_{n_i < n <
n_{i+1}} \BOX_n|_P = n_{i+1} - n_i - 1$.
Since $\HEIGHT_{n_0} < \HEIGHT_{n_i}$ for all $i > 0$, it follows that
$|[n_0,n_{i+1})\times[0,\HEIGHT_{n_0}]|_P < n_{i+1} - n_0$ for all $i >
0$. Since $\HEIGHT_{n_0} > 1$, this is a contradiction to condition
\hyperlink{SIPWIT3}{7}.
\end{pf}

\begin{thmm}\label{thminfinitelymany}
If $n_{\mathcal I} > -\infty$, $m_{\mathcal I}=\infty$ and $P$ is
$\ssi
$-distinguished,
then $\BOXES(P,\ssi)$ contains
infinitely many components, all of which are finite. Furthermore, for
any given component, if $n$ is the rightmost integer belonging to the component,
then $\HEIGHT_n>1$ and the set of vertices of the component is $\{
\LEFT
_n+1,\ldots,n\}$.
\end{thmm}
\begin{pf}
Let $P$ satisfy the hypothesis of the theorem. We construct a~sequence
of integers $\{n_i\}_{i=0}^\infty$ as follows. Let $n_0 = n_{\mathcal
I}$. For $i \in\N$, let $n_{i+1}$ be the largest integer greater than
$n_i$ such that $\BOX_{n_{i+1}}$ contains the point $(n_i,0)$. We must
now show this sequence is well defined. Suppose not, and let $i$ be
minimum such that there is no valid choice for $n_{i+1}$. Then there
are infinitely many integers $n >n_i$ such that $B_{n_i} \subseteq
B_n$. Since $i$ was chosen minimum, for each such $n$ we have $\LEFT
_n=n_i$. By Lemma~\ref{lemboxgraph1}, it must be that for each such
$n$, $\HEIGHT_n < 1$. But this implies that $|[n_i,n)\times[0,1]|_P
\ge
n - n_i$ for all $n \in\Z\cap[n_i,\infty)$, a contradiction to
condition \hyperlink{RWdies}{8}. Thus the sequence $\{n_i\}
_{i=0}^{\infty}$ is
well defined.

For all $i\in\N$, it follows from the definition of $n_{i+1}$ that
$\LEFT_{n_{i+1}}=n_i$, and there is no integer $n>n_{i+1}$ for which
$\BOX_{n_{i+1}} \subseteq\BOX_n$; thus $\BOX_{n_{i+1}}$ and $\BOX
_{n_{i+2}}$ are in separate components of $\BOXES(P,\ssi)$. By
Lemma~\ref{lemboxcontainment}, all $\BOX_n$ such that $n_i < n \le
n_{i+1}$ are contained in $\BOX_{n_{i+1}}$ and hence in the same
component of $\BOXES(P,\ssi)$.

If $\HEIGHT_{n_{i+1}} < 1$ for some $i \in\N$, then by Lemma \ref
{lemboxgraph2}, there are infinitely many $n \in\Z\cap
(n_{i+1},\infty
)$ such that $\BOX_{n_{i+1}} \subseteq\BOX_n$, but this is a
contradiction to the choice of $n_{i+1}$.
By Lemma~\ref{lemleftboxtaller}, we have for all $i \in\N$ that
$\HEIGHT_{n_{i+1}} > \HEIGHT_{n_{i+2}}$. We conclude that $\HEIGHT
_{n_{i+1}} > 1$ for all $i \in\N$.
\end{pf}

Given an interval $\ssi$, with $n_{\mathcal I} > -\infty$ and
$m_{\mathcal I}=\infty$, and an $\ssi$-distinguished point set $P$,
define $\{n_i\}_{i=0}^\infty= \{n_i(P,\mathcal I)\}_{i=0}^\infty$ as
in the proof of Theorem~\ref{thminfinitelymany}.
\begin{cor}\label{corbackbone}
If $n_{\mathcal I} > -\infty$, $m_{\mathcal I}=\infty$ and $P$ is
$\ssi
$-distinguished, then $\IPC(P,\ssi)$ is a tree that consists of a
unique infinite backbone (i.e., a unique, infinite, self-avoiding path
originating from the root) from which emerge finite branches.
Furthermore, the backbone contains the points $\{n_i\}_{i=0}^{\infty}$.
\end{cor}
\begin{pf}
Clearly, $\IPC(P,\ssi)$ is acyclic, and is connected by Lemma \ref
{lemboxpath}, so is a tree.
For each integer $i \geq1$, let $\Phi_i$ be the unique path from $n_i$
to $n_0=n_{\ssi}$ in $\IPC(P,\ssi)$. By Lemma~\ref{lemboxpath},
it follows that for all $i \geq1$, $\Phi_i$ is a sub-path of~$\Phi
_{i+1}$, and so the limit $\Phi= \lim_{i \to\infty} \Phi_i$ is a
well-defined infinite path starting from $n_{\ssi}$. Furthermore, for
any integer $k$, any path $\Phi'$ starting from
$v_k$, that is, edge-disjoint from $\Phi$ must have all its elements
among $n_i,\ldots,n_{i+1}-1$, where $n_i \leq v_k < n_{i+1}$. Thus, all
branches leaving $\Phi$ are finite.
\end{pf}

\begin{remark*}
In general, relaxing any of the conditions in the definition of
distinguished point sets may cause the conclusions of Corollary \ref
{corbackbone} to fail. To provide just one example,
the following point set satisfies conditions
\hyperlink{IPWIT1}{1}--\hyperlink{SIPWIT3}{7},
but not the conclusion of Corollary~\ref{corbackbone}. The
point set $P$ contains no points except the following. Place $k$ points
inside $[0,1)\times[0,1/3)$. Place each of the points $(i+1,1-1/(i+2))$
for $i \in\N$. Then $\IPC(P,\N)$ has $k$ infinite backbones.
\end{remark*}

This completes our study of deterministic properties of the invasion
percolation procedure.
In the next section we begin our study of what happens when the
underlying point set is random.

%s3 ###
\section{Invasion percolation on the PWIT}\label{secinvadepwit}
Throughout Section~\ref{secinvadepwit}, $\pff$ denotes a Poisson
process of constant intensity $1$ in $[0,\infty)\times[0,\infty)$, so
$\pff$ is almost surely \mbox{$\N$-distinguished}.
By Corollary~\ref{ipcpoisson}, $\IPC(\pff,\N)$ is distributed as the
invasion percolation cluster $\st_0$ of the PWIT, so results for $\IPC
(\pff,\N)$ apply to~$\st_0$ mutatis mutandis.
Below, we will derive more precise statements about the structure of
$\IPC(\pff,\N)$ than can be made under the assumptions of Section
\ref
{secpwitpp}. First, however,
we state two ``ballot-style'' theorems for stochastic processes that we
will use repeatedly.

%s3.1 ###
\subsection{Two ballot-style theorems}
The following result was proved independently by Tanner \cite
{tanner61derivation} and Dwass~\cite{dwass62fluctuation}.
\begin{lem}[(Cycle lemma)]
Suppose that $X_1,\ldots,X_n$ are integer-valued, cyclically
interchangeable random variables with maximum value $1$.
Then for any integer $0 \leq k \leq n$,
\[
\mathbf{P}\{ S_i > 0\ \forall1 \leq i \leq n \mid S_n=k \} = \frac{k}{n}.
\]
\end{lem}

The next result was proved by T\'akacs~\cite{takacs67comb}, page 12.
\begin{lem}[(Stationary ballot theorem)]\label{statballot}
Let $X_1,X_2,\ldots$ be an infinite sequence of i.i.d. integer random
variables with mean $\mu$ and maximum value~$1$, and for any $i \geq
1$, let $S_i = X_1+\cdots+X_i$.
Then
\[
{\mathbf{P}}\bigl\{S_n > 0\ \forall n\in\{1,2,\ldots\}\bigr\} =
\cases{
\mu,& \quad$\mbox{if } \mu> 0,$ \vspace*{2pt}\cr
0, & \quad$\mbox{if } \mu\le0.$}
\]
\end{lem}

Now let $P$ be a random point set in, say, $[m,n]\times[0,\infty)$.
We recall the definition of the condition $F(m,n,y)$ from Lemma \ref
{lemequivalent}, and will abuse notation by also writing $F(m,n,y,P)$ for
the \textit{event} that the \textit{condition} $F(m,n,y,P)$ holds.
[At times
we write $F(m,n,y)$ in place of $F(m,n,y,P)$, when $P$ is clear from
context.]
Notice that if $P$ is a uniform set of $n-1$ points in $[0,n]\times
[0,\lambda)$, for some $\lambda>0$, then
applying the cycle lemma with $X_i= 1-|[i-1,i)\times[0,\lambda)|_P$
(and so\vadjust{\goodbreak} $S_i= i-|[0,i)\times[0,\lambda)|_P$) for $i=1,\ldots,n$, it
follows that the probability that $F(0,n,\lambda,P)$ occurs is
precisely $1/n$.
By an argument of a similar nature, we can straightforwardly derive the
following lemma (which can also be deduced from an existing result
(\cite{mcdiarmid97mst}, Theorem~4) for invasion percolation on $K_n$
and a limiting argument).
\begin{lem}\label{uniform}
Fix an integer $n \geq1$, and list the $n$ elements of $\pff\cap
([0,n] \times[0,\infty))$ of lowest height as $q_1,\ldots,q_n$, in
increasing order of height.
Then for each $i=1,\ldots,n$, ${\mathbf{P}}\{p_n=q_i\}=1/n$.
\end{lem}

We emphasize that the $n$ elements of $\pff\cap([0,n] \times
[0,\infty
))$ of lowest height may not all be elements of the set $\{p_1,\ldots
,p_n\}$, or indeed of the set $\{p_i\}_{i=1}^{\infty}$.
\begin{pf*}{Proof of Lemma \protect\ref{uniform}}
Fix $n$, and let $\lambda=y(q_n)$.
Clearly, $p_n$ will be among $q_1,\ldots,q_n$.
Also, let $\ssi= \{0,\ldots,n\}$, let $P=\{q_1,\ldots,q_{n}\}$ and for
each $i\in\{0,\ldots,n-1\}$, let $P^i=\{q_1^i,\ldots,q_n^i\}$ be the cyclic
shift of $P$ to the right by distance $i$.
Then for all $i \in\{1,\ldots,n-1\}$, $P^i$ is distributed as $n-1$
uniform points in $[0,n]\times[0,\lambda)$, together with a single uniform
point of height $\lambda$.
We claim that with probability $1$, for each $j=1,\ldots,n$, there is
exactly one $i=i(j,P) \in\{0,\ldots,n-1\}$ for which $p_n(P^i,\ssi)=q_j^i$.
Since the $P^i$ are identically distributed it follows from this claim that
\[
{\mathbf{P}}\{p_n(P,\ssi)\!=\!q_j\} = \frac{1}{n} \sum_{i=0}^{n-1}
{\mathbf{P}}\{p_n(P^i,\ssi)\!=\!q_j^i\} = \frac{1}{n} {\mathbf{P}}\Biggl\{
\bigcup_{i=0}^{n-1}\{p_n(P^i,\ssi)\!=\!q_j^i\}\Biggr\}=\frac{1}{n},
\]
which proves the theorem. It thus remains to prove the above claim,
which we do by contradiction. Thus, suppose that
for some $j \in\{1,\ldots,n\}$, there are distinct $i, i' \in\{
0,\ldots,n-1\}$ for which $p_n(P^i,\ssi)=q_j^i$ and $p_n(P^{i'},\ssi
)=q_j^{i'}$.\vspace*{-2pt}
By replacing $P$ by either $P^{n-i}$ or $P^{n-i'}$ if necessary, we may
assume that $i'=0$. Let $q_j=q_j^0=(x_j,y_j)$.
We must have $|[n-i,n]\times[0,y_j)|_P< i$ [or else $p_n(P,\ssi) \neq
q_j$]; on the other hand, $|[\LEFT_n(P,\ssi),n]\times
[0,y_j]|_P=n-\LEFT
_n(P,\ssi)$.

Let $k=n-\LEFT_n(P,\ssi)$, the length of $\BOX_n(P,\ssi)$.
If $k \geq i$, then we also have $|[\LEFT_n,n-i]\times[0,y_j]|_P \geq
n-\LEFT_n-i+1$, so
$|[\LEFT_n+i,n]\times[0,y_j]|_{P^i} \geq n-\LEFT_n-i+1$ and $\HEIGHT
_n(P^i,\ssi)<y_j$, contradicting the fact that $p_n(P^i,\ssi)=q_j^i$.
It follows that $k < i$, that is, that $i-(n-\LEFT_n) \geq1$.
In this case, we have that for each $m\in\{1,\ldots,i-(n-\LEFT_n)\}$,
$[\LEFT_n-m,\LEFT_n]\times[0,y_j] < m$ (or else we would have either
chosen $h_n$ lower or $\ell_n$ smaller).
Translating the above information to~$P^i$, we see that $|[i-k,i]\times
[0,y_j]|_{P^i}=k$, that
$|[i-k,i]\times[0,y_j)|_{P^i} <k$, and that $|[i-k',i]\times
[0,y_j]|_{P^i} < k'$ for each $k' \in\{k+1,\ldots,i\}$ (see Figure
\ref
{cyclicshift}).
Thus, by Definition~\ref{dfnboxes}\vspace*{-3pt} and Lemma~\ref{boxesproc},
$p_i(P^i,\ssi) = q_j^i$, contradicting the assumption that
$p_n(P^i,\ssi
)=q_j^i$.
\end{pf*}
%
%%
%f3 ###
\begin{figure}

\includegraphics{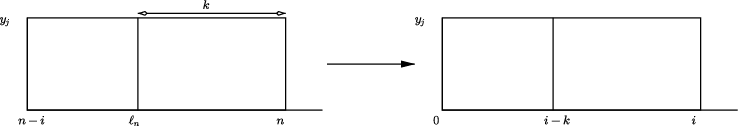}

\caption{The region on the left maps onto the region on the right when
$P$ is sent to $P^i$.}
\label{cyclicshift}
\end{figure}

%%
%%\input{test.pdf_t}
%%
%%
%%
%%
%{12}{14.4}{\rmdefault}{\mddefault}{\updefault}{\color[rgb]{0,0,0}$k$}%
%}}}}
%{12}{14.4}{\rmdefault}{\mddefault}{\updefault}{
%}}}}
%{12}{14.4}{\rmdefault}{\mddefault}{\updefault}{
%}}}}
%{12}{14.4}{\rmdefault}{\mddefault}{\updefault}{\color[rgb]{0,0,0}$0$}%
%}}}}
%{12}{14.4}{\rmdefault}{\mddefault}{\updefault}{\color[rgb]{0,0,0}$i$}%
%}}}}
%{12}{14.4}{\rmdefault}{\mddefault}{\updefault}{
%}}}}
%{12}{14.4}{\rmdefault}{\mddefault}{\updefault}{\color[rgb]{0,0,0}$\ell
%_n$}%
%}}}}
%{12}{14.4}{\rmdefault}{\mddefault}{\updefault}{
%}}}}
%{12}{14.4}{\rmdefault}{\mddefault}{\updefault}{\color[rgb]{0,0,0}$n$}%
%}}}}
%%
%}
%%

We next elaborate on a connection between Poisson Galton--Watson trees
and queueing theory that will be useful for many subsequent calculations.

%s3.2 ###
\subsection{A fact from queuing theory and an aside on Poisson
Galton--Watson duality}\label{queueing}
The following basic result was first noted by Borel~\cite{borel1942etb}.
Consider a~queue with Poisson rate $\lambda$ arrivals and constant,
unit service time, started at time zero with a single customer in the
queue, and with any arbitrary servicing rule (i.e., not necessarily
first-in first-out). We may form a rooted tree associated with the
queueing process run until the first time $\tau$ that there are no
customers in the queue (or forever, if the queue is never empty), in
the following manner.
If a new customer joins the queue at time~$t$, he is joined to the
customer being served at time~$t$.
We denote the resulting rooted tree by $\st$. Then $\st$ is distributed
as a Poisson($\lambda$) Galton--Watson tree [we write $\PGW(\lambda)$,
for short]~\cite{borel1942etb}.

If the arrival times are given by the $x$-coordinates of the points of
Poisson process $\mathcal{Q}=\pff\cap([0,\infty)\times[0,\lambda))$,
we may also associate an interpolated random walk to the process, by setting
$S_t=|[0,t)\times[0,\lambda)|_{\mathcal Q}-t$ for $t \in\mathbb{R}^+$.
Then~$|\st|$ is simply the first time $t$ that $S_t = -1$, that is,
that $|[0,t)\times[0,\lambda)|_{\mathcal Q}=t-1$.
Note that given that $|T|=m< \infty$, $\mathcal{Q}$ is distributed as
$m-1$ independent uniform points in $[0,t)\times[0,\lambda]$,
conditioned on $F(0,n,\lambda,\mathcal{Q})$ occurring.
We also observe that for all $i \leq|\st|$, $S_i=S^{0,\lambda^-}_i$,
where $S^{0,\lambda^-}$ is the random walk defined in Section~\ref
{secrw}. This has immediate implications for the events defined in
Section~\ref{secrw}.
In particular, the tree $\mathcal T$ is infinite if and only if
$S^{0,\lambda^-}$ survives.
It follows that for all $0 \le h \le1$ we have ${\mathbf{P}}\{
S^{k,\lambda^-} \mbox{ survives}\}=0$, and for all $\lambda>1$ we
have ${\mathbf{P}}\{S^{k,\lambda^-} \mbox{ survives}\}=\theta
(\lambda)$,
the probability of survival of a $\PGW(\lambda)$ branching process.
Similarly, by Lemma~\ref{statballot}, we have that the probability that
there is ever a time~$t$ at which the total number of arrivals is at
least $t$, is $\min(\lambda,1)$. Thus,
if $0 < \lambda\le1$, then ${\mathbf{P}}\{S^{k,\lambda^-} \mbox{ has a
chance}\}=\lambda$, and if $\lambda> 1$ then ${\mathbf{P}}\{
S^{k,\lambda^-} \mbox{ has a chance}\}=1$.
Of course, the exact same identities hold with $S^{k,\lambda^-}$
replaced by $S^{k,\lambda}$, $L^{k,\lambda^-}$ or $L^{k,\lambda}$.

We continue to think of arrival times as given by points of $\mathcal{Q}$.
It will be useful for us to view the above queuing procedure as
creating a tree whose nodes are labeled by integers rather than by
elements of the queue. We do so by re-labeling each node of $\st$
(i.e., each customer $c$) with the (integer) time at which $c$ begins
being served. Furthermore, suppose that we take as our servicing rule
the invasion percolation rule---that is, the rule that prioritizes
customers (points of $\mathcal Q$) with lower $y$-coordinate over those
with higher $y$-coordinate---and call the resulting tree $\st_{\lambda}$.
Then $\st_{\lambda}$ is precisely the subtree of $\IPC(\mathcal
P,\mathbb N)$ containing the root and all nodes joined to the root by
paths all of whose edges have weight less than $\lambda$. Of course,
everything still holds if we take
$\mathcal{Q}=\pff\cap([0,\infty)\times[0,\lambda])$---that is,
if we
include points at height precisely $\lambda$---as long as we replace
$S^{0,\lambda-}$ by $S^{0,\lambda}$ and replace the phrase ``less than
$\lambda$'' by ``at most $\lambda$.''

As a consequence of the above discussion we have the following
important fact.
\begin{lem}\label{ipcpgw}
Fix any integer $n \geq1$, any $\lambda> 0$, and let $P$ be a set of
$n-1$ independent uniform points in $[0,n]\times[0,\lambda)$.
Given that $F(0,n,\lambda,P)$ occurs,
the tree $\IPC(P,\{0,\ldots,n-1\})$ is distributed as $\PGW(\lambda)$
conditioned to have $n$ nodes. Furthermore, suppose that $p$ is a
uniformly random point on the line segment $[0,n]\times\{\lambda\}$.
Then under the same conditioning, $\IPC(P\cup\{p\},\{0,\ldots,n\})$ is
distributed as $\PGW(\lambda)$ conditioned to have $n$ nodes, together
with an additional vertex (vertex $n$), joined to a
uniformly random element of $0,\ldots,n-1$.
\end{lem}

We remark that for any fixed $n$, the distribution of $\PGW(\lambda)$
conditioned to have~$n$ nodes does not depend on $\lambda$ and is
precisely that of a uniformly random labelled rooted tree (or Cayley
tree) on~$n$ nodes, after the labels but not the orders of children
have been discarded.
It turns out that a version of Lemma~\ref{ipcpgw} also holds for the
box tree; see Lemma~\ref{lemboxespgw}, below.
As noted just after Lemma~\ref{statballot}, the probability that
$F(0,n,\lambda,P)$ occurs is precisely $1/n$.
Thus, the distribution of $|\pgw(\lambda)|$ is given by
\[
{\mathbf{P}}\{|\pgw(\lambda)|=n\}=\frac{1}{n}{\mathbf{P}}\{\operatorname
{Poisson}(\lambda n)=n-1\}=\frac{e^{-\lambda n}(\lambda n)^{n-1}}{n!},
\]
for all positive integers $n$ (a well-known fact which we record for
later reference). When $\lambda=1$ this is called the \textit{Borel}
distribution.

We briefly explain a further basic fact about the function $\theta
(\lambda)=\break{\mathbf{P}}\{|\pgw(\lambda)|=\infty\}$ and about
Poisson Galton--Watson duality.
By considering the number of children in the first generation of $\pgw
(\lambda)$, we see that $1-\theta(\lambda) = e^{-\lambda\theta
(\lambda
)}$, and
by differentiating this identity, we see that
%
%
%e1 ###
\begin{equation}\label{theta}
\theta'(\lambda)\bigl(1-\lambda\bigl(1-\theta(\lambda)\bigr)\bigr) =
\theta(\lambda
)\bigl(1-\theta(\lambda)\bigr),
\end{equation}
an equation we will have use of later.
Next, given $\lambda> 1$, let $m=m(\lambda) < 1$ be such that
$\lambda
e^{-\lambda}=me^{-m}$ (we call $m$ the \textit{dual parameter} for
$\lambda
$). Then $m=\lambda(1-\theta(\lambda))$, from which it is easily
seen that conditional on being finite, $\pgw(\lambda)$ is distributed
precisely as $\pgw(m)$.

%s3.3 ###
\subsection{$\IPC(\pff,\N)$ and the forward maximal process}\label{fmp}
By Corollary~\ref{corbackbone},\break $\IPC(\pff,\N)$ consists of a unique
infinite backbone which in particular passes through the nodes $\{n_i\}
_{i=0}^{\infty}$, and from all nodes\vadjust{\goodbreak} of which emerge finite branches.
Let the edges of the backbone be $e_1,e_2,\ldots,$ and for each integer
$i \geq1$ let $W_i = \sup_{j \geq i} W_{e_j}$, so $\{W_i\}
_{i=1}^{\infty}$ is the PWIT forward maximal process.
From the perspective of the PWIT, the nodes $n_i$ are the nodes at
which the forward maximal weight along the backbone decreases.

Lemma~\ref{ipcpgw} allows us to provide another picture of the
structure of\break
$\IPC(\pff,\N)$. First, for each integer $i \geq0$, let
$T_i=T_i(\pff,\N
)$ be the subtree of $\IPC(\pff,\N)$ on nodes $n_{i},\ldots,n_{i+1}-1$
(these nodes induce a tree by Lemma~\ref{lemboxpath}). The set
$P=\pff\cap([n_{i},n_{i+1})\times[0,\HEIGHT_{n_{i+1}}))$ is
distributed as $(n_{i+1}-n_{i}-1)$ independent uniform points,
conditional on $F(n_{i},n_{i+1},\HEIGHT_{n_{i+1}},P)$ occurring.
Furthermore, $\pff\cap([n_i,n_{i+1})\times\{\HEIGHT_{n_{i+1}}\})$
contains a single uniform point. Thus, by Lemma~\ref{ipcpgw}, we obtain
that $T_i$ is distributed as
$\pgw(\HEIGHT_{n_{i+1}})$ conditioned to have $n_{i+1}-n_i$ nodes, and
that $n_{i+1}$ is joined to a uniformly random element of $T_i$.
Applying this for all~$i$, we obtain the following theorem.

\begin{thmm}\label{conditionalipc}
Given $\{n_i\}_{i=0}^{\infty}$, $\IPC(\pff,\N)$, viewed as an unlabeled
tree, can be built as follows. For each integer $i \geq0$ let $T_i$ be
a uniformly random labeled tree on $n_i-n_{i-1}$ vertices.
For each integer $i \geq1$, join the root of $T_i$ to a uniformly
random vertex of $T_{i-1}$. Finally, discard all labels.
\end{thmm}

\begin{remark*} It also follows straightforwardly from Lemma \ref
{ipcpgw} that given $\{n_i\}_{i=0}^{\infty}$ and $\{\HEIGHT
_{n_{i+1}}\}
_{i=0}^{\infty}$, $\IPC(\pff,\N)$ viewed as a \textit{weighted}
unlabeled tree
can be built from the tree described in Theorem~\ref{conditionalipc} as
follows. Independently for each integer $i \geq0$ and each edge $e$ of
$T_i$, assign $e$ a random weight with
$\operatorname{Uniform}[0,\HEIGHT_{n_{i+1}}]$ distribution. Also, for each
integer $i \geq0$, give the unique edge from $T_{i+1}$ to $T_i$ the
weight $\HEIGHT_{n_{i+1}}$. We omit the details.
\end{remark*}

We next show that $\{(|T_{i-1}|,\HEIGHT_{n_i})\}_{i=1}^{\infty}=\{
(n_i-n_{i-1},\HEIGHT_{n_i})\}_{i=1}^{\infty}$ is a Markov process and
specify the transition probabilities.
First, for any $i \geq1$, given~$\HEIGHT_{n_{i}}$, the set
$\mathcal{Q}=\pff\cap([n_i,\infty)\times[0,\HEIGHT_{n_i}))$ is
precisely a Poisson point process\vspace*{-2pt} of intensity $1$ conditioned
on the event that $S^{n_i,\HEIGHT_{n_i^-}}(\mathcal{Q})$ survives
(which is precisely the event that $\mathcal{Q}$ is
$\{n_i,n_{i+1},\ldots,\}$-reasonable). Furthermore, given~$\HEIGHT
_{n_i}$, the condition $G(n_i,\HEIGHT_{n_i}^-)$ holds for $\pff$.
Thus, by Lemma~\ref{lemrestrict},\vspace*{1pt} we can determine the structure of
$\IPC(\pff,\N)$ restricted to
$\{n_i,\ldots,\infty\}$ by considering only the points in $\mathcal{Q}$.
It follows that $\{(n_i-n_{i-1},\HEIGHT_{n_i})\}_{i=1}^{\infty}$ is a
Markov process, as claimed (and also that $\{\HEIGHT_{n_i}\}
_{i=1}^{\infty}$ is a Markov process).
Next, for $1 < y < h$, let
\[
f_h(y) = \lim_{dy \to0} \frac{\mathbf{P}\{ \HEIGHT_{n_{i+1}} \in
dy \mid\HEIGHT_{n_i}=h \}}{dy},
\]
and for $n > 0$ let
\[
f_h(n,y) = \lim_{dy \to0} \frac{\mathbf{P}\{ \HEIGHT_{n_{i+1}}
\in dy,(n_{i+1}-n_i)=n \mid\HEIGHT_{n_i}=h \}}{dy},
\]
so $f_h(y) = \sum_n f_h(n,y)$. By the above comments, $f_h(y)$ and
$f_h(n,y)$ do not depend on $i$.\vadjust{\goodbreak}
\begin{lem}\label{forwarddensities}
For all $i \geq1$ and $1 < y < h$,
\[
f_h(y)=\frac{\theta'(y)}{\theta(h)} \quad\mbox{and} \quad f_h(n,y) =
\frac{\theta(y)}{\theta(h)} \frac{e^{-yn } (y n)^{n-1}}{(n-1)!}.
\]
\end{lem}

Combining the two results in Lemma~\ref{forwarddensities}, the
following corollary is immediate.
\begin{cor}\label{forwardsizes}
For all integers $n,i \geq1$ and $y > 1$,
%
%
%e2 ###
\begin{equation}\label{forwardsizeeq}
\mathbf{P}\{ n_{i+1}-n_i=n \mid\HEIGHT_{n_{i+1}}=y \} = \frac{\theta
(y)}{\theta
'(y)} \frac{e^{-ny} (ny)^{n-1}}{(n-1)!}.
\end{equation}
\end{cor}
\begin{pf*}{Proof of Lemma \protect\ref{forwarddensities}}
We have
\begin{eqnarray*}
\mathbf{P}\{ \HEIGHT_{n_{i+1}} \leq y \mid\HEIGHT_{n_i}=h \}
& =& \mathbf{P}\{ S^{n_i,y}(\pff)\mbox{ survives} \mid S^{n_i,h^-}(\pff
)\mbox{ survives} \} \\
& =& \frac{{\mathbf{P}}\{S^{n_i,y}(\pff)\mbox{ survives}\}}{{\mathbf
{P}}\{S^{n_i,h^-}(\pff)\mbox{ survives}\}} \\
& =& \frac{\theta(y)}{\theta(h)},
\end{eqnarray*}
and the first claim of the lemma follows by differentiation.

As mentioned, $f_h(y)$ and $f_h(n,y)$ do not depend on $i$, so we take
$i=0$ (and thus $n_i=0$).
In order to have $n_1-n_0=n$ and $\HEIGHT_{n_1} \in dy$, we need that
$|[0,n]\times[0,y)|_{\pff}=n-1$, that $|[0,n]\times[y,y+dy)|_{\pff}=1$,
that $F(0,n,y,\pff)$ occurs
and that $S^{n,y^-}(\pff)$ survives. The probabilities of the first two
events are easily bounded. The probability of $F(0,n,y,\pff)$ given that
$|[0,n]\times[0,y)|_{\pff}=n-1$ is
$1/n$ by the cycle lemma. Finally, the event that $S^{n,y^-}(\mathcal
{P})$ survives is independent of the first three events, and has
probability $\theta(y)$. Thus,
\begin{eqnarray*}
&&\mathbf{P}\{ \HEIGHT_{n_{i+1}} \in dy,(n_{i+1}-n_i)=n \mid\HEIGHT
_{n_i}=h \}\\
&&\qquad=
\frac{e^{-yn } (y n)^{n-1}}{(n-1)!} \cdot\bigl(1+o(dy)\bigr) n \,dy
\cdot\frac
{1}{n} \cdot\theta(y) \cdot\frac{1}{\theta(h)},
\end{eqnarray*}
from which the second claim of the lemma follows.
\end{pf*}

We next derive the distribution of the distance along the backbone
between $n_{i-1}$ and $n_{i}$.
For $i \geq0$ let $d_i=d_i(\pff,\ssi) = d_{\IPC}(n_i,n_{i+1})$. As with
the quantities studies above, we have that given $\HEIGHT_{n_{i+1}}$,
$d_i$ is independent of the past.\vspace*{-2pt}
For $0 < x < 1$ we say $X\eqdist\operatorname{Geometric}(x)$ if
${\mathbf
{P}}\{X=k\}=x^k(1-x)$.
\begin{thmm}\label{geom}
For all $i \geq0$ and all $y > 1$, given that $\HEIGHT_{n_{i+1}}=y$,
$d_i \eqdist1+ \operatorname{Geometric}(m(y))$.
\end{thmm}

The following theorem derives the distributions of the trees hanging
off the backbone and within a given pond.\vadjust{\goodbreak}
\begin{thmm}\label{offbackbone}
Fix $i \geq0$ and $y > 1$. Given that $\HEIGHT_{n_{i+1}}=y$,
for all $k$ with $n_i \leq k < n_{i+1}$ and for which $k$ is on the
backbone, the subtree of $\IPC(\pff,\N)$ containing $k$ and
containing no other vertices of the backbone, is distributed as
$\pgw(m(y))$.\vspace*{-4pt}
\end{thmm}

Together, Theorems~\ref{geom} and~\ref{offbackbone} provide another
Markovian characterization of $\IPC(\pff,\N)$:
we may construct a tree with the distribution of $\IPC(\pff,\N)$ by
growing the trees hanging off the backbone one-by-one,
where the branching distribution of the trees depends on the current
forward maximal weight.
(This characterization is exactly that which is claimed in Theorem \ref
{main3}, which is proved below.)
The forward maximal weight process
evolves according to the dynamics implied by Lemma \ref
{forwarddensities} and Theorem~\ref{geom}:
first stay constant for a geometric amount of time depending on the
current forward maximal weight, then decrease the maximal weight
according to Lemma~\ref{forwarddensities}.
This characterization is essentially the $\sigma\to\infty$ limit of
results of Angel et~al.~\cite{angel2006ipr}
described in the \hyperref[secintro]{Introduction}, for
invasion percolation on the regular $\sigma$-ary tree. However, it does
not seem trivial
to derive these results from theirs by a limiting argument and local
weak convergence, since they depend not only on the graph structure of
the tree but also on the weights.

If one wishes, at this point one can apply all the methodology of \cite
{angel2006ipr} to see that corresponding results hold for invasion
percolation on the PWIT: notably, convergence to the Poisson lower envelope,
mutual singularity of IPC and IIC measures and spectral asymptotics all
hold for invasion percolation on the PWIT. We have not included the
details as the development is essentially technical, requiring no significant
ideas not already found in~\cite{angel2006ipr}.

Since $m(\lambda)=\lambda(1-\theta(\lambda))$, we also obtain the
following corollary of Theorems~\ref{conditionalipc},~\ref{geom}
and~\ref{offbackbone} and Corollary~\ref{forwardsizes}, which is new as
far as we know.
For any $0 < p < 1$, let $D = \operatorname{Geometric}(p)$, and let
$r=v_0,v_1,\ldots,v_D=v$ be a path of length $D$. For each $k\in\{
0,\ldots,D\}e$,
starting from $v_k$ grow a $\PGW(p)$ tree with root $v_k$. This yields
a triple $(T,r,v)_p$.\vspace*{-4pt}
\begin{cor}
Fix any $0 < p < 1$ and let $(T,r,v)_p$ have the distribution described
above. Then conditional on $|T|$, $T$ is distributed as $\PGW(p)$
conditioned to have size $|T|$, and $v$ is distributed as
a uniformly random node of $T$. Furthermore, let $y=m^{-1}(p)$. Then
for all $n \geq1$, ${\mathbf{P}}\{|T|=n\}$ is given by the right-hand
side of (\ref{forwardsizeeq}).\vspace*{-4pt}
\end{cor}

We now turn to the proof of Theorem~\ref{geom}.
Recall that for $y > 1$, $m(y)$ is the dual parameter for $y$.
We will make use of the following easy (and known) lemma.\vspace*{-4pt}
\begin{lem}\label{cayley}
Let $T$ be a Cayley tree of order $n$ with root $r$, and let $v$ be a
uniformly random node in $T$. Then
\[
{\mathbf{P}}\{d_T(r,v)=k-1\} = \frac{k (n)_{k}}{n^{k+1}}\vspace*{-3pt}
\]
where $(n)_{k}=n(n-1)\cdots(n-k+1)$.\vspace*{-3pt}\vadjust{\goodbreak}
\end{lem}
\begin{pf}
We view $T$ as a doubly-rooted tree with roots $r$ and $v$.
By removing the edges on the path from $r$ to $v$, we obtain a forest
of $d_T(r,v)$ rooted trees, whose roots are ordered (as $r=v_1,\ldots
,v_k=v$, say).
The number of such forests is $({n \atop k}) k! \cdot(k n^{n-k-1})$;
see, for example,~\cite{moon70counting}, Theorem 3.2. The result
follows by dividing by the
total number of doubly-rooted trees on $n$ labeled vertices, which
is~$n^n$ by Cayley's formula.
\end{pf}

\begin{pf*}{Proof of Theorem \protect\ref{geom}}
As mentioned at the start of Section~\ref{fmp}, for all $i$, $n_{i+1}$
is joined to a uniformly random element of $T_{i}$---say $v_i$---and $d_{\IPC}(n_i,n_{i+1})=1+d_{T_i}(n_i,v_i)$.
By Corollary~\ref{forwardsizes} and Lemma~\ref{cayley}, it follows that
\begin{eqnarray*}
& &\mathbf{P}\{ d_{\IPC}(n_i,n_{i+1})=k \mid\HEIGHT_{n_{i+1}}=y \} \\
&&\qquad= \sum_{n=k}^{\infty} \mathbf{P}\{ d_{T_i}(n_i,v_i)=k-1 \mid|T_i|=n
\}\mathbf{P}\{ |T_i|=n \mid\HEIGHT_{n_{i+1}}=y \} \\
&&\qquad= \sum_{n=k}^{\infty} \frac{k (n)_{k}}{n^{k+1}} \cdot\frac
{\theta
(y)}{\theta'(y)} \frac{e^{-ny} (ny)^{n-1}}{(n-1)!} \\
&&\qquad= y^{k-1}\frac{\theta(y)}{\theta'(y)} \sum_{n=k}^{\infty} \frac{k}{n}
\cdot\frac{e^{-ny} (ny)^{n-k}}{(n-k)!} \\
&&\qquad= y^{k-1}\frac{\theta(y)}{\theta'(y)} \sum_{n=k}^{\infty}\frac{k}{n}
{\mathbf{P}}\{\operatorname{Poisson}(ny)=n-k\}.
\end{eqnarray*}
By the cycle lemma,
\[
\frac{k}{n}{\mathbf{P}}\{\operatorname{Poisson}(ny)=n-k\} = {\mathbf
{P}}\{
S_n^{0,y}=-k, S_i^{0,y} > -k\ \forall0 \le i < n\},
\]
so
\[
\mathbf{P}\{ d_{\IPC}(n_i,n_{i+1})=k\mid\HEIGHT_{n_{i+1}}=y \} =
y^{k-1}\frac
{\theta(y)}{\theta'(y)} {\mathbf{P}}\{S_n^{0,y}=-k\mbox{ for some
}n\}.
\]
But by the connection with queueing theory explained above,
${\mathbf{P}}\{S_n^{0,y}=-k\mbox{ for}\mbox{some }n\}$ is
precisely the
probability that
$k$ independent
$\pgw(y)$ all fail to survive, which is $(1-\theta(y))^k$. We complete
the proof by applying the identity~(\ref{theta}) from page \pageref{theta}.
\end{pf*}

\begin{pf*}{Proof of Theorem \protect\ref{offbackbone}}
For this theorem we revert to viewing\break $\IPC(\pff,\N)$ as a subtree of
the PWIT $\st$; our proof is based on the proof of Proposition 2.3
of~\cite{angel2006ipr}.
Given a node $v \in\st$, write $\st_v$ [resp. $\st_v(\lambda)$] for
the subtree of $\st$ rooted at $v$ (resp. rooted at $v$ and containing
all edges of weight at most $\lambda$ to descendants of $v$)---so $\st_v(\lambda) \eqdist\pgw(\lambda)$---and write $\lambda
^*(v) =
\inf\{\lambda\dvtx\st_v(\lambda)\mbox{ is
infinite}\}$.\vadjust{\goodbreak}

Let $r$ be the root of $\st$, and fix any node $v \in\st$. Fix $y > 1$
and integers $1 \leq j \leq k$. Let $E_{v,j,k,dy}$ be the event that $v$
is on the backbone, has $k$ children of whom the backbone passes
through the $j$th (in the left-to-right ordering of the PWIT)
and $\lambda^*(v) \in dy$. We split $E_{v,j,k,dy}$ into four events
depending on distinct edge sets of the PWIT:
\begin{enumerate}[$F_1$]
\item[$F_1$] $\st_r(y)-\st_v(y)$ is finite. (This event depends only on
edges of $\st-\st_v$.)
\item[$F_2$] $v$ has precisely $k$ children of weight at most $y+dy$---say $v_1,\ldots,v_k$. (This depends only on the weights of edges from
$v$ to its children.)
\item[$F_3$] For $i \in\{1,\ldots,k\}\setminus\{j\}$, $\st
_{v_i}(y+dy)$ is finite. (This depends only on the weights edges in the
subtrees $\st_{v_i}$ for $i \neq j$.)
\item[$F_4$] $\st_{v_j}(y)$ is finite, but $\st_{v_j}(y+dy)$ is
infinite. (This depends only on the weights of edges in $\st_{v_j}$.)
\end{enumerate}
Since the edge sets determining the events $F_1,\ldots,F_4$ are
disjoint, if $E$ occurs, then the conditioning on subtrees $\st
_{v_i}(y)$ for $i \in\{1,\ldots,k\}\setminus\{j\}$
is precisely that they are finite. Thus, given that $E$ occurs, for
each $i \in\{1,\ldots,k\}\setminus\{j\}$, $\st_{v_i}(y)$ is
distributed as $\pgw(m(y))$.

Now let $E_{v,dy}= \bigcup_{i,j} E_{v,i,j,dy}$---so $E_{v,dy}$ is the
event that $v$ is on the backbone and that $\lambda^*(v) \in dy$. Given
the observation at the end of the
previous paragraph, to prove the theorem it suffices
to show that as $dy \to0$, given $E_{v,dy}$, the number $N_v(y+dy)$ of
children of $v$ in $\st_v(y+dy)$ approaches $\operatorname{Poisson}(m(y))+1$
in distribution
[so that the number of children off the backbone approaches
$\operatorname
{Poisson}(m)$].
To see this is an easy calculation. First,
\begin{eqnarray*}
{\mathbf{P}}\{E_{v,dy}\} & = &{\mathbf{P}}\{F_1\}\cdot{\mathbf{P}}\{
\st_v(y)\mbox{ is finite but }\st_v(y+dy)\mbox{ is infinite}\} \\
& =& {\mathbf{P}}\{F_1\} \cdot\bigl(1+o(dy)\bigr) \theta'(y)\,dy.
\end{eqnarray*}
Next, fixing $k \geq1$, by symmetry,
\begin{eqnarray*}
&&{\mathbf{P}}\{N_v(y+dy)=k,E_{v,dy}\}\\
&&\qquad= \bigl(1+o(dy)\bigr) {\mathbf{P}}\{
F_1\} \cdot{\mathbf{P}}\{\operatorname{Poisson}(y)=k\} \cdot k \cdot
\bigl(1-\theta(y)\bigr)^{k-1} \cdot\theta'(y) \,dy.
\end{eqnarray*}
The factor $k$ above selects which of the $k$ children of $v$ is on the
backbone.
Since $m=y(1-\theta(y))$, it follows that
\begin{eqnarray*}
\lim_{dy \to0} \mathbf{P}\{ N_v(y+dy)=k\mid E_{v,dy} \}& =& {\mathbf
{P}}\{\operatorname{Poisson}(y)=k\} \cdot k \cdot\bigl(1-\theta(y)\bigr
)^{k-1}\\
& =&
{\mathbf{P}}\{\operatorname{Poisson}(m)=k-1\},
\end{eqnarray*}
which completes the proof.
\end{pf*}

%s4 ###
\section{The stationary graph and box processes}\label{secstatproc}
Throughout this section, $\pff$ denotes a Poisson point process of
intensity $1$ in the upper half-plane, so $\pff$ is almost surely
exemplary and $\Z$-distinguished.

%s4.1 ###
\subsection{Rooted subtrees in the box tree}
Recall that $\BOXES=\BOXES(\pff,\Z)$ is the tree with vertex set
$\Z$
defined in Section~\ref{secpwitpp}. For $n \in\Z$, we let $\BOXES^n$
denote the subtree of $\BOXES$ rooted at $n$, and write $|\mathrm{BG}^n|$
for its size (number of vertices).
By the definition of $\BOXES$, it is immediate that $|\mathrm{BG}
^n|=n-\LEFT
_n$. Our main aim in this section
is to prove the following theorem.
\begin{thmm}\label{boxesiic}
$\BOXES(\pff,\Z)$ is distributed as the Poisson IIC, in the local
weak sense.
\end{thmm}

A key step in proving Theorem~\ref{boxesiic}, one that additionally
introduces several of the main ideas, is the following theorem.
\begin{thmm}\label{boxespgw}
For all $n \in\Z$, conditional on $\LEFT^n(\pff,\Z)$ and $\HEIGHT
^n(\pff
,\Z)$, $\BOXES^n(\pff,\Z)$ is distributed as $\pgw(\HEIGHT^n)$
conditioned to have $n-\LEFT^n$ nodes. Furthermore, unconditionally
$\BOXES^n$ is distributed as $\pgw(1)$.
\end{thmm}
\begin{cor}[(\cite{mcdiarmid97mst}, Theorem 1)]
We have $\lfloor x(p_0) \rfloor\eqdist-\lfloor A V\rfloor$, where $A$
is Borel distributed, and $V$ is $\operatorname{Uniform}[0,1]$ and
independent of $A$.
\end{cor}
\begin{pf}
Given $\LEFT_0$ and $\HEIGHT_0$, the line segment $[\LEFT_0,0]\times
\{
\HEIGHT_0\}$ contains a single uniformly random point,
and this point is $p_0$. The second assertion of the theorem implies
that $\ell_0$ is Borel distributed, and the corollary follows.
\end{pf}

For the next several pages, we focus on developing the tools needed for
the proof of Theorem~\ref{boxespgw}.
By translation invariance, it suffices to prove Theorem~\ref{boxespgw}
with $n=0$.
We prove the theorem by way of the following analog of Lemma \ref
{ipcpgw} that holds for the box tree.
\begin{lem}\label{lemboxespgw}
Let $n$, $\lambda$, $P$ and $p$ be as in Lemma~\ref{ipcpgw}. Given that
$F(0,n,\lambda,P)$ occurs, $\BOXES(P\cup\{p\},\{0,\ldots,n\})$ is
distributed as $\PGW(\lambda)$ conditioned to have~$n$ nodes.
\end{lem}

We remark that for any $\lambda>0$, $\PGW(\lambda)$ conditioned to
have $k$ nodes and $\PGW(1)$ conditioned to have $k$ nodes are
identically distributed.
Thus, in proving Theorem~\ref{boxespgw} and Lemma~\ref{lemboxespgw} we
may and shall at times assume without loss of generality that $\lambda=1$.
Figure~\ref{manyboxes} contains an example of $\BOXES(P\cup\{p\},\{
0,\ldots,n\})$ for~$P$, $p$ as in Lemma~\ref{lemboxespgw}, with $n=256$.
(By the preceding comment, the value of~$\lambda$ is not important.)
In proving the lemma, it will be important to view $\pgw(1)$ both as an
ordered (plane) tree and as an unordered tree.
The ordered perspective is natural for $\pgw(1)$ when viewed as a
subtree of the PWIT.
Next, fix an unordered, rooted tree $U$. We will abuse notation by
writing $\pgw(1)=U$ if
$\pgw(1)=T$ for some ordered tree $T$ with underlying unordered tree $U$.
Fix one such tree $T$, and let $\operatorname{aut}(U)$ be the number of
rooted automorphisms of $T$.
[Note: by this we mean the number of distinct plane trees with\vadjust{\goodbreak}
underlying unrooted tree $U$; e.g., for the
tree $U$ in Figure~\ref{planetree}, interchanging the pair of leaves~$x$
and $y$ does not affect the
plane tree, and $\operatorname{aut}(U)=12$.] We then have ${\mathbf
{P}}\{\pgw(1)=U\}
= \operatorname{aut}(T)\cdot{\mathbf{P}}\{\pgw(1)=T\}$.
%

%
%
%f4 ###
\begin{figure}

\includegraphics{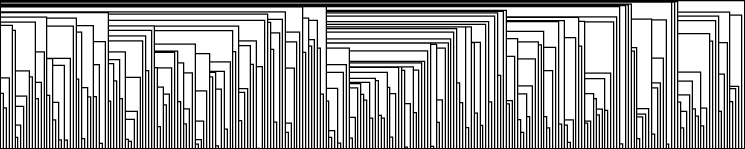}

\caption{The boxes for a random set of 256 points. The points
themselves are omitted. The code for generating this image was written
by Omer Angel.}
\label{manyboxes}
\end{figure}

%

%%
%f5 ###
\begin{figure}[b]

\includegraphics{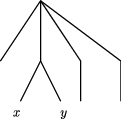}

\caption{Viewed as an unordered tree, the above tree has $
\mathrm{aut}(U)=12$.}
\label{planetree}
\end{figure}

%%
%%\input{tree.pdf_t}
%%
%%
%%
%%
%{12}{14.4}{\rmdefault}{\mddefault}{\updefault}{\color[rgb]{0,0,0}$y$}%
%}}}}
%{12}{14.4}{\rmdefault}{\mddefault}{\updefault}{\color[rgb]{0,0,0}$x$}%
%}}}}
%%
%}

%%

We next turn our attention to $\BOXES^0$. There is again a natural
ordering of children in $\BOXES^0$---vertices are integers, and when
we refer to a box tree as an ordered tree,
we are referring to the ordering inherited from the integers.
However, unlike in $\pgw(1)$, we cannot expect the distributions of
distinct subtrees to be identical under this ordering.
Given an unordered tree $U$, we will also abuse notation by writing
${\mathbf{P}}\{\BOXES^0 = U\}$ if $\BOXES^0$ is unlabeled, rooted
isomorphic to $U$.

Our proof of Lemma~\ref{lemboxespgw} makes use of the following easy fact.
\begin{lem}\label{perm} Let $r\ge2$ and let $s_{1},\ldots, s_{r}$ be
natural numbers. Then
\[
\sum_{\pi} \prod_{j=2}^{r}\frac{s_{\pi(j)}}{\sum
_{i=1}^{j}s_{\pi(i)}}=1 ,
\]
where the summation is over all permutations $\pi$ of $\{1,\ldots,r\}$.
\end{lem}

\begin{pf} We proceed by induction on $r$. For $r=2$, the sum is
over just two permutations, and the result is
\[
\frac{s_{1}}{s_{2}+s_{1}}+\frac{s_{2}}{s_{1}+s_{2}}=1
,
\]
as required. For general $r$, we partition the set of permutations $\pi
$ of $\{1,\ldots,r\}$ depending on the value of $\pi(r)$---for each
$k=1,\ldots,r$, let $\Pi_{k}$ be the set of permutations\vadjust{\goodbreak} $\pi$ of $\{
1,\ldots,r\}$ with $\pi(r)=k$. Since our aim is to prove that the sum
\[
\sum_{\pi} \prod_{j=2}^{r}\frac{s_{\pi(j)}}{\sum
_{i=1}^{j}s_{\pi(i)}} = \sum_{k=1}^{r}\sum_{\pi\in\Pi_{k}}\prod
_{j=2}^{r}\frac{s_{\pi(j)}}{\sum_{i=1}^{j}s_{\pi(i)}}
\]
has the value one, it suffices to prove that for each $k$ we have
\[
\sum_{\pi\in\Pi_{k}}\prod_{j=2}^{r}\frac{s_{\pi(j)}}{\sum
_{i=1}^{j}s_{\pi(i)}}=\frac{s_{k}}{\sum_{i=1}^{r}s_{i}} .
\]
Since the expression on the right-hand side here is the $j=r$ term of
the product for all $\pi\in\Pi_{k}$, it suffices to show that
\[
\sum_{\pi\in\Pi_{k}}\prod_{j=2}^{r-1}\frac{s_{\pi(j)}}{\sum
_{i=1}^{j}s_{\pi(i)}}=1 .
\]
By re-labeling if necessary, this may be deduced from the induction hypothesis.
\end{pf}

\begin{pf*}{Proof of Lemma \protect\ref{lemboxespgw}}
Fix an unordered tree $U$ with $n$ vertices and root $r$.
We will show that
%
%
%e3 ###
\begin{equation}\label{boxewpgwtoprove}
n \cdot{\mathbf{P}}\{\BOXES(P,[n])=U\} = \mathbf{P}\{ \pgw(1)=U \mid
|\pgw(1)|=n \}.
\end{equation}
Proving this equality will prove the lemma, since $F(0,n,\lambda,P)$
must occur in order to have $\BOXES(P,[n])=U$, and ${\mathbf{P}}\{
F(0,n,\lambda,P)\}=1/n$ as noted just
after Lemma~\ref{statballot}. The case $n=1$ of~(\ref{boxewpgwtoprove})
is trivial, so suppose that $n>1$ and that the proposition holds for
all $n'$ with $1 \leq n' < n$.

Order the children of the root $r$ of $U$ arbitrarily, and suppose that
the subtrees
$U_1,\ldots,U_k$ of $U$ rooted at the children of the root have sizes
$n_1,\ldots,n_k$
with respect to this order. Let $\operatorname{aut}(r)$ be the number of
permutations of the children of $r$ which induce automorphisms of $U$.
[For example, the tree in Figure~\ref{planetree} has
$\operatorname{aut}(r)=2$.]

We note that
%
%e4 ###
\begin{eqnarray}\label{pgwcalc}
&& \mathbf{P}\{ \pgw(1)=U \mid|\pgw(1)|=n \}\nonumber\\
&&\qquad= \frac{{\mathbf{P}}\{\pgw(1)=U\}}{{\mathbf{P}}\{|\pgw{1}|=n\}}
\nonumber\\
&&\qquad= \frac{n!}{n^{n-1}e^{-n}} \cdot\frac{e^{-1}}{k!}\frac
{k!}{\operatorname
{aut}(r)} \prod_{i=1}^k {\mathbf{P}}\{\pgw(1)=U_i\} \\
&&\qquad= \frac{1}{\operatorname{aut}(r)}\frac{n!}{n^{n-1}}
\prod_{i=1}^k \frac{n_i^{n_i-1}}{n_i!}\mathbf{P}\{ \pgw(1)=U_i \mid
|\pgw(1)=n_i \}
\nonumber\\
&&\qquad= \frac{n}{\operatorname{aut}(r)}\pmatrix{n-1 \vspace*{2pt}\cr
n_1,\ldots,n_k}
\prod_{i=1}^k\biggl (\frac{n_i}{n}\biggr)^{n_i}\frac{1}{n_i}\mathbf{P}\{ \pgw
(1)=U_i \mid|\pgw(1)=n_i \}.\nonumber
\end{eqnarray}
Next, assume the points of $P$ are listed in increasing order of height
as $\{p_1,\ldots, p_n\}$.
We first consider the sizes of the subtrees of $\BOX_0$.

Let $m_0=0$ and for $i=1,\ldots,k$, let $m_i=m_{i-1}+n_i$ (so in
particular $m_k=n-1$).
Also for $i=1,\ldots,k$, let $Q_i$ be the set of points $p \in
P\setminus\{p_n\}$ satisfying
$m_{i-1} \leq x(p) < m_i$, and let $p^i$ be the point of $Q_i$ with
greatest $y$-coordinate.

We recall the definitions of the events $E,F$ and $G$ from Lemma \ref
{lemequivalent}.
In order for $n$ to have children with subtrees $U_1,\ldots,U_k$ in
that order,
it is necessary and sufficient that the following events occur:

\begin{longlist}[(III)]
\item[(I)] $E(m_{i-1},m_{i},y(p^i),P)$ occurs for each $i\in\{
1,\ldots
,k\}$;
\item[(II)] We have $y(p^i) > y(p^2) > \cdots> y(p^k)$;
\item[(III)] $F(m_{i-1},m_{i},y(p^i),P)$ and $G(m_{i-1},y(p^i),P)$
occur for each $i=\{1,\ldots, k\}$;
\item[(IV)] $\BOXES^{m_i}=U_i$ for each $i=\{1,\ldots,k\}$.
\end{longlist}
First, (I) is equivalent to the requirement that $|Q_i|=n_i$ for each
$i=\{1,\ldots,k\}$.
The $x$-coordinates of points in $P$, are uniformly distributed on $[-n,0]$
so
%
%
%e5 ###
\begin{equation}\label{event1}
{\mathbf{P}}\{\mbox{(I)}\} = \pmatrix{n-1 \vspace*{2pt}\cr n_1,\ldots
,n_k}\prod
_{i=1}^k \biggl(\frac{n_i}{n}\biggr)^{n_i}.
\end{equation}
Given (I), for (II) to occur it suffices that for each $i=\{1,\ldots
,k\}$,
the point of $(P\setminus\{p_n\})\setminus(Q_1\cup\cdots\cup Q_{i-1})$
with the largest $y$-coordinate, is
a member of $Q_{i}$. Thus,
%
%
%e6 ###
\begin{equation}\label{event2}
\mathbf{P}\{ \mbox{(II)} \mid \mbox{(I)} \} = \prod_{i=1}^{k} \frac
{n_i}{n-1-\sum
_{j=1}^{i-1}n_j}.
\end{equation}
Since $y(p^i)<y(p^j)$ for $j < i$, if (I), (II) and $\bigcap
_{i=1}^{k-1}F(m_{i-1},m_{i},y(p^i))$ all hold for some $k \geq1$, then
it is immediate that
$G(m_{i-1},y(p^i))$ occurs for each $i=\{1,\ldots,k\}$.
Thus,
\[
\mathbf{P}\{ \mbox{(III)} \mid \mbox{(I)},\mbox{(II)} \} = \mathbf
{P}\Biggl\{ \bigcap_{i=1}^k F(m_{i-1},m_{i},y^*(i)) \Bigm| \mbox
{(I)},\mbox
{(II)} \Biggr\}.
\]
Furthermore, given (I) and (II), independently for each $i=\{1,\ldots
,k\}$, the points of $Q_i\setminus\{p^i\}$ are independently and
uniformly distributed in $[m_{i-1},m_i)\times[0,y(p^i))$.
Thus, by the cycle lemma,
\[
\mathbf{P}\{ \mbox{(III)} \mid \mbox{(I)},\mbox{(II)} \} = \prod
_{i=1}^k \mathbf{P}\{ F(m_{i-1},m_{i},y(p^i)) \mid\mbox{(I)},\mbox
{(II)} \} \nonumber
= \prod_{i=1}^k \frac{1}{n_i}.\label{event3}
\]
Finally, given (I), (II) and (III) and independently for each $i=\{
1,\ldots,k\}$, $Q_i\setminus\{p^i\}$ is precisely a uniform
set of $n_i-1$ points, conditioned on $F(m_{i-1},m_i,\break y(p^i),Q_i)$
holding, and $p^i$ is a uniform point on $[m_i-n_i,m_i]\times\{y(p^i)\}$.
Furthermore, by Lemma~\ref{lemrestrict}, given\vadjust{\goodbreak}
$E(m_{i-1},m_{i},y(p^i))$, $F(m_{i-1},m_{i},y(p^i))$ and $G(m_{i-1},y(p^i))$,
we have
\[
\BOXES^{m_i}(P,\{0,\ldots,n\})= \BOXES(Q_i,\{m_{i-1},\ldots,m_i\}).
\]
Thus, by the induction hypothesis,
%
%e7 ###
\begin{eqnarray}\label{event4}
&& \mathbf{P}\{ \mbox{(IV)} \mid\mbox{(I)},\mbox{(II)},\mbox{(III)}
\} \nonumber\\
&&\qquad= \prod_{i=1}^k \mathbf{P}\{ \BOXES^{m_i}(P,\{0,\ldots,n\} )=U_i
\mid E(m_{i-1},m_{i},y(p^i)),\nonumber\\
&&\qquad\hspace*{86pt}F(m_{i-1},m_{i},y(p^i)),G(m_{i-1},y(p^i)) \}
\\
&&\qquad= \prod_{i=1}^k \mathbf{P}\bigl\{ \BOXES(Q_i,\{m_{i-1},\ldots,m_i\}
)=U_i \mid F(m_{i-1},m_{i},y(p^i)) \bigr\} \nonumber\\
&&\qquad= \prod_{i=1}^k \mathbf{P}\{ \pgw(1)=U_i \mid|\pgw(1)|=n_i
\}.\nonumber
\end{eqnarray}
Combining~(\ref{event1})--(\ref{event4}), and rearranging, we obtain that
\begin{eqnarray*}
&& {\mathbf{P}}\bigl\{n\mbox{ has children }U_1,\ldots,U_k\mbox{ in that
order in }\BOXES(P\cup p,\{0,\ldots,n\}\bigr\} \\
&&\qquad=
\pmatrix{n-1 \vspace*{2pt}\cr n_1,\ldots,n_k}\prod_{i=1}^{k} \frac
{n_i}{n-1-\sum
_{j=1}^{i-1}n_j}\\
&&\hspace*{72pt}\qquad{}\times\prod_{i=1}^k \biggl(\frac{n_i}{n}\biggr
)^{n_i} \frac
{1}{n_i}\mathbf{P}\{ \pgw(1)=U_i \mid|\pgw(1)|=n_i \}, \\
&&\qquad= \frac{\operatorname{aut}(r)}{n}\prod_{i=1}^{k} \frac
{n_i}{n-1-\sum
_{j=1}^{i-1}n_j} \mathbf{P}\{ \pgw(1)=U \mid|\pgw(1)|=n \},
\end{eqnarray*}
the latter equality holding due to~(\ref{pgwcalc}).
To obtain ${\mathbf{P}}\{\BOXES(P\cup p,\{0,\ldots,n\})=U\}$, we now
must sum this
bound over distinct orderings of $U_1,\ldots,U_k$. We instead sum over
all permutations $\pi\dvtx[k]\to[k]$, and note that this counts each
distinct ordering
$\operatorname{aut}(r)$ times. We thus obtain
\begin{eqnarray*}
&& {\mathbf{P}}\bigl\{\BOXES(P\cup p,\{0,\ldots,n\})=U\bigr\} \\
&&\qquad= \frac{1}{n}\mathbf{P}\{ \pgw(1)=U \mid|\pgw(1)|=n \}\cdot\sum
_{\pi\dvtx[k]\to[k]}
\prod_{i=1}^{k} \frac{n_{\pi(i)}}{n-1-\sum_{j=1}^{i-1}n_{\pi(j)}}.
\end{eqnarray*}
By Lemma~\ref{perm}, the above sum is 1, which establishes (\ref
{boxewpgwtoprove}) by induction and so completes the proof.
\end{pf*}

In proving Theorem~\ref{boxespgw}, we will use the following identity,
which we quote in advance.
\begin{lem} \label{integrate} For integers $a \geq0$, $b > 0$, let
$I_{a,b}:= \int_{0}^{1}x^{a}e^{-bx} \,dx$. Then
$I_{b-1,b}-I_{b,b}=e^{-b}/b$ for each $b> 0$.
\end{lem}
\begin{pf}
Integration by parts.
\end{pf}

The final step before proving Theorem~\ref{boxespgw} is to derive the
conditional distribution of $\LEFT_0=|\mathrm{BG}^0|$ given $\HEIGHT_0$.
As this will be useful later in the paper, we state it as a separate lemma.
Write
\[
\varphi_y(n) = \lim_{dy \to0} \frac{{\mathbf{P}}\{|\LEFT
_0|=n,\HEIGHT_0 \in[y,y+dy)\}}{dy}.
\]
\begin{lem}\label{heightlengthcond}
For all $0 < y < 1$ and all $n \geq1$,
%
%
%e8 ###
\begin{equation}\label{heightlength}
\varphi_y(n) = (1-y)\cdot\frac{e^{-ny}(ny)^{n-1}}{(n-1)!}.
\end{equation}
\end{lem}
\begin{pf}
Fix $n\in\N$ and $0 < y < 1$.
In order to have $\ell_0 = -n$ and $\HEIGHT_0=y$, it is necessary and
sufficient that
$E=E(-n,0,y,P),F=F(-n,0,y,P)$ and $G=G(-n,y,P)$, from Lemma \ref
{lemequivalent} all occur.
We first calculate the density of the event $E$.
\begin{eqnarray*}
& &{\mathbf{P}}\{|[-n,0)\times[0,y)|_{\pff}=n-1,|[-n,0)\times
[y,y+dy)|_{\pff}=1\} \\
&&\qquad= \bigl(1+o(dy)\bigr) {\mathbf{P}}\{\operatorname
{Poisson}(ny)=n-1\}\cdot n \,dy \\
&&\qquad= \bigl(1+o(dy)\bigr) \frac{e^{-ny}(ny)^{n-1}}{(n-1)!} \cdot n
\,dy.
\end{eqnarray*}
Now let $f_E(y) = f_E(0,n,y) = \frac{e^{-ny}(ny)^{n-1}}{(n-1)!} \cdot n$.
Given that\vspace*{1pt} $|[-n,0)\times[0,y)|_{\pff}=n-1$ occurs, $\pff
([-n,0]\times
[0,y))$ consists of $n-1$ uniformly random points.
Independently of this, given that $|[-n,0)\times\{y\}|_{\pff}=1$,
the line segment $[-n,0]\times\{y\}$ contains a single uniformly
random point.
By the first of the two preceding observations and by the cycle lemma,
it follows that $\mathbf{P}\{ F \mid E \}=\frac{1}{n}$.
Furthermore, $G$ is independent of $E$, and so by Lemma \ref
{statballot}, ${\mathbf{P}}\{G\}=1-y$.
We thus have
\[
\varphi_y(n) =\mathbf{P}\{ F,G \mid E \} f_E(y) = {\mathbf{P}}\{G\}
\mathbf{P}\{ F \mid E \} f_E(y),
\]
from which the lemma follows.
\end{pf}
\begin{pf*}{Proof of Theorem \protect\ref{boxespgw}}
We assume without loss of generality that $n=0$.
Let $P = ([\LEFT_0,0]\times[0,\HEIGHT_0])\cap\pff$, and let $p =
([\LEFT
_0,0] \times\{\HEIGHT_0\}) \cap\pff$.
Then $P$ is precisely distributed as a set of $|\LEFT_0|-1$ uniform
points in $([\LEFT_0,0]\times[0,\HEIGHT_0])$,
conditional on $F(\LEFT_0,0,\HEIGHT_0)$, and $p$ has uniform
distribution on $([\LEFT_0,0] \times\{\HEIGHT_0\})$.
The\vadjust{\goodbreak} first claim of the theorem then follows from Lemma~\ref{lemboxespgw}.
Next, by Lemma~\ref{heightlengthcond}, for any positive integer $m$ we have
%
%e9 ###
\begin{eqnarray}\label{length}
{\mathbf{P}}\{|\LEFT_0|=m\} & =& \int_{0}^{1} \varphi_y(m) \,dy
\nonumber
\\[-8pt]
\\[-8pt]
\nonumber
& =& \frac{m^{m-1}}{(m-1)!}( I_{m-1,m}-I_{m,m}),
\end{eqnarray}
where the notation $I_{a,b}$ is that defined in Lemma~\ref{integrate}.
Applying that lemma, we obtain that ${\mathbf{P}}\{|\mathrm{BG}^0|=m\}$ is
$m^{m-1}e^{-m}/m!$ [exactly the probability that a PGW(1) has size $m$].
The second claim of the theorem then follows from the first and the
fact that the conditional distribution of $\pgw(\lambda)$ given its
size, is independent of $\lambda$.
\end{pf*}

Before proving Theorem~\ref{boxesiic}, we first state a consequence of
the above development.
\begin{cor}
For any positive integer $n$, any $0 < y < 1$, and any unordered rooted
tree $U$ with $|U|=n$,
\[
\lim_{dy \to0} \frac{{\mathbf{P}}\{\BOXES^0=U, \HEIGHT_0 \in dy\}
}{dy} = (1-y) n {\mathbf{P}}\{\pgw(y)=U\}.
\]
\end{cor}
\begin{pf}
Immediate from~(\ref{heightlength}) and Theorem~\ref{boxespgw}.
\end{pf}

\begin{pf*}{Proof of Theorem \protect\ref{boxesiic}}
We will in fact prove that for any unordered rooted tree $U'$,
conditional on $\BOXES^0=U'$,
$\BOXES^{a_0}\setminus\BOXES^0$ is distributed as $\pgw(1)$, which
implies the statement of the theorem.
Thus, let $U$ and $U'$ be unordered rooted trees with roots $r$ and $r'$,
and let $U^*$ be the unordered rooted tree with root $r$ obtained by
adding an edge between $r$ and $r'$.
We define
\begin{eqnarray*}
E & =& \{\BOXES^0=U',\BOXES^{a_0}\setminus\BOXES^0=U\} \\
& = &\{\BOXES^0=U',\BOXES^{a_0}=U^*\}.
\end{eqnarray*}
Next let $k=k(U^*)$ be the number of children of $r$ in $U^*$, and let
$j=j(U^*,U') \geq1$ be the number of children of $r$ in $U^*$ whose
subtree is isomorphic to $U'$.
Also, let $\operatorname{aut}(r)$ [resp. $\operatorname{aut}^*(r)]$ be
the number
of permutations of the children of $r$ in $U$ (resp. $U^*$) which
induce automorphims of $U$ (resp.~$U^*$).
Note that $\operatorname{aut}^*(r)=j\cdot\operatorname{aut}(r)$.
%
%%
%f6 ###
\begin{figure}

\includegraphics{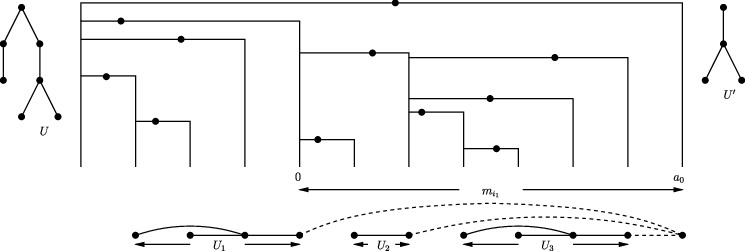}

\caption{In the above figure, an ordering ${\mathbf u}=(U_1,U_2,U_3)$
of the children of $U^*$ is fixed. Dashed edges lead from $a_0$ to its
three children.
In this example, $j(U^*,U')=2$ since $U_1$ and $U_3$ are (unordered,
rooted) isomorphic to $U'$ but $U_2$ is not, and $i_1({\mathbf u})=1$,
$i_2({\mathbf u})=3$ for the same reason. Finally,
this example relates to the event $E_{{\mathbf u},p}$ with $p=1$,
since $a_0=m_{i_1}({\mathbf u})$.}
\label{thm30}
\end{figure}

%%
%%\input{thm30.pdf_t}
%%
%%
%%
%%
%{12}{14.4}{\rmdefault}{\mddefault}{\updefault}{\color
%[rgb]{0,0,0}$m_{i_1}$}%
%}}}}
%{12}{14.4}{\rmdefault}{\mddefault}{\updefault}{\color[rgb]{0,0,0}$U$}%
%}}}}
%{12}{14.4}{\rmdefault}{\mddefault}{\updefault}{\color[rgb]{0,0,0}$U'$}%
%}}}}
%{12}{14.4}{\rmdefault}{\mddefault}{\updefault}{
%}}}}
%{12}{14.4}{\rmdefault}{\mddefault}{\updefault}{\color[rgb]{0,0,0}$0$}%
%}}}}
%{12}{14.4}{\rmdefault}{\mddefault}{\updefault}{
%}}}}
%{12}{14.4}{\rmdefault}{\mddefault}{\updefault}{
%}}}}
%{12}{14.4}{\rmdefault}{\mddefault}{\updefault}{
%}}}}
%%
%}

%%
%

Given an ordering $ {\mathbf u} =(U_1,\ldots,U_k)$ of the children of
$r$ in $U^*$, let $i_1=i_1(\mathbf{u}),\ldots,i_j=i_j(\mathbf{u})$ be
the indices $i$ for which $U_{i}$ is isomorphic to $U'$.
For each $\ell\in\{0,\ldots,k\}$, let $m_{\ell}=m_\ell(\mathbf
{u})=1+\sum_{q=\ell+1}^k |U_q|$ (which is $1$ when $\ell=k$). Let
\[
E_{ {\mathbf u}} = E \cap\{ a_0 \mbox{ has children } U_1,\ldots,U_k
\mbox{ in that order}\},
\]
and for each $p\in\{1,\ldots,j\}$, let
\[
E_{ {\mathbf u} , p } = E_{\mathbf u} \cap\{a_0=m_{i_p}(\mathbf{u})\}
= \{ m_{i_p}(\mathbf{u}) \mbox{ has children } U_1,\ldots,U_k \mbox{
in that order}\}.\vadjust{\goodbreak}
\]
(The above definitions are depicted in Figure~\ref{thm30}.)
Then $E=\bigcup_{\mathbf u} \bigcup_{p=1}^{j} E_{ {\mathbf u} , p }$,
where the first union is over all distinct orderings
$\mathbf{u}$ [we say $\mathbf{u}=(U_1,\ldots,U_k)$ and $\mathbf
{u}'=(U_1',\ldots,U_k')$ are distinct if there
is some $i \in\{1,\ldots,k\}$ for which $U_i$ and $U_i'$ are not
isomorphic]. The double union is then over disjoint terms, and so
\[
{\mathbf{P}}\{E\} = \sum_{\mathbf u} \sum_{p=1}^j {\mathbf{P}}\{
E_{\mathbf{u},p }\}.
\]
Reversing the order of summation above, by translation invariance and
Theorem~\ref{boxespgw}, we obtain
\begin{eqnarray*}
{\mathbf{P}}\{E\} & =& \sum_{p=1}^j \sum_{\mathbf u} {\mathbf{P}}\{
E_{\mathbf{u},p}\} \\
& = &j \cdot{\mathbf{P}}\{\BOXES^0= U^*\} \\
& =& j \cdot{\mathbf{P}}\{\pgw(1)=U^*\}.
\end{eqnarray*}
Now fix some ordering $U_1,\ldots,U_k$ of the children of $r$ in $U^*$
with $U_k = U'$. Then by the definition of $\pgw(1)$ and the preceding
equality, we have
\begin{eqnarray*}
{\mathbf{P}}\{E\} & = &j \cdot\frac{e^{-1}}{k!} \frac{k!}{\operatorname
{aut}^*(r)} \prod
_{i=1}^k {\mathbf{P}}\{\pgw(1)=U_i\} \\
& =& {\mathbf{P}}\{\pgw(1)=U'\} \cdot\frac{e^{-1}}{(k-1)!}\frac
{(k-1)!}{\operatorname
{aut}(r)} \prod_{i=1}^{k-1} {\mathbf{P}}\{\pgw(1)=U_i\} \\
& =& {\mathbf{P}}\{\pgw(1)=U'\}{\mathbf{P}}\{\pgw(1)=U\}.
\end{eqnarray*}
It follows by Theorem~\ref{boxespgw} that
\begin{eqnarray*}
\mathbf{P}\{ \BOXES^{a_0}\setminus\BOXES^0=U | \BOXES^0=U' \} & =&
\frac{{\mathbf{P}}\{E\}}{{\mathbf{P}}\{\BOXES^0=U'\}} \\
& = &{\mathbf{P}}\{\pgw(1)=U\},
\end{eqnarray*}
proving the theorem.
\end{pf*}

%s4.2 ###
\subsection{An ancestral process in $\IPC(\pff,\Z)$}

By the end of this section we will have proved Theorems~\ref{main1}--\ref{main3} from the \hyperref[secintro]{Introduction}.
To warm up, we prove the following theorem.\vadjust{\goodbreak}
\begin{thmm}\label{pgw1}
The subtree of $\IPC(\pff,\Z)$ rooted at zero and containing only nodes
with positive label, is distributed as $\pgw(1)$.
\end{thmm}
\begin{pf}
By Proposition~\ref{propunder1}, each point of $\pff$ in $[0,1]\times
[0,1]$ yields a~child of $0$ in $T_0$, so $0$ has $\operatorname{Poisson}(1)$
children.
Let $r_0=0$, and let $r_1 = \min(i > 0 \dvtx\lfloor x(p_i)\rfloor
=0)$, so
$r_1$ is ``the first integer with $0$ as a parent.'' For any $i >0$,
the event that $r_1=i$ is independent of
$\pff\cap( [i,\infty)\times[0,\infty) )$, so $r_1$ also has
$\operatorname
{Poisson}(1)$ children. More generally, let $r_{k}=\min(i >
r_{k-1} \dvtx\lfloor x(p_i)\rfloor\in\{r_0,\ldots,r_{k-1}\})$. Then
for all $k\geq1$ and all $i>0$, the event that $r_k=i$ is independent
of $\pff\cap( [i,\infty)\times[0,\infty) )$, so $r_k$ has $\operatorname
{Poisson}(1)$ children.
The nodes $r_0,r_1,r_2,\ldots$ are precisely the descendants of $0$,
and we have just seen that each has $\operatorname{Poisson}(1)$ children
independently of all the others. This proves the theorem.
\end{pf}

Heuristically, the fact that $\IPC(\pff,\Z)$ is equal in
distribution to
the IIC can be seen as follows. By symmetry, from Theorem~\ref{pgw1},
at each node of $\IPC(\pff,\Z)$ is rooted
a copy of $\pgw(1)$.
Also, by exploring the nodes of multiple trees in a left-to-right
fashion as in Theorem~\ref{pgw1}, we see that the offspring
distribution for distinct branches of $\IPC(\pff,\Z)$ are independent.
Furthermore, the parent of $0$ in $\IPC(\pff,\Z)$ is more likely to
be a
node with many children than one with few children. This should
``size-bias'' the number of children
of the parent of zero, in such a way as to precisely compensate for the
edge from $0$ to its parent, so that a $\operatorname{Poisson}(1)$
number of
children remain. The same argument should also hold for the parent of the
parent of zero, and so on ad infinitum. It is possible to make (parts
of) this heuristic argument rigorous; however, we
obtain the result as a relatively direct byproduct of our argument for
Theorem~\ref{main3}, whose proof requires a different approach.

The key to the proof is the definition of a ``backward maximum
process'' which is extremely similar to the forward maximal process.
We begin by listing $0$ and its ancestors in $\IPC(\pff,\Z)$ in
decreasing order as $n_0,n_1,n_2$, et cetera, so in particular $n_0=0$
and in general $n_{i+1}=\lfloor x(p_{n_i}) \rfloor$.
For $i \geq0$ let $w_i = \HEIGHT_{n_i}$, and let\vadjust{\goodbreak} $m_i = \max_{0 \leq j
\leq i} w_j$, the greatest weight of any of the first $(i+1)$ edges. In
particular, $m_0=w_0=\HEIGHT_0$.
Finally, let $i_0=0$ and, for $k \geq1$, let $i_k$ be the smallest
integer $i > i_{k-1}$ for which $m_{i_k} > m_{i_{k-1}}$.
Then the following lemma is basic (but important).
\begin{lem}\label{leftheightsame}
For all $k \geq1$, $n_{i_k} = \LEFT(n_{i_{k-1}})$, and so
$m_{i_k}=\HEIGHT_{n_{i_k}}$.
\end{lem}
\begin{pf}
The fact that $n_{i_k} \leq\ell(n_{i_{k-1}})$ is immediate from
Lemma~\ref{lemboxcontainment}. But $\ell= \ell(n_{i_{k-1}})$ is an
ancestor of $n_{i_{k-1}}$ by Lemma~\ref{lemboxpath}, and $h_{\ell} >
h_{n_{i_{k-1}}}$ by Lemma~\ref{lemleftboxtaller}. Thus, $\ell(n_{i_{k-1}}) = n_{i_k}$ as claimed.
\end{pf}

Above, we derived the joint distribution of the height and length of
$\BOX_0$. We next show that the sequence $\{m_{i_k}\}_{k \in\N}$ has a
particularly simple and pleasing description.
\begin{lem}\label{uniformdist}
The sequence $\{m_{i_k}\}_{k \in\N}$ is a homogeneous Markov chain,
and for all $k$, given $m_{i_k}$, $m_{i_{k+1}}$ has distribution
$\operatorname{Uniform}[m_{i_k},1]$.
\end{lem}
\begin{pf}
For all $0 < y < 1$ and all $k$, $\HEIGHT_k \leq y$ if and only if the
random walk $L^{k,y}$ has a chance. As remarked in Section~\ref{queueing},
the probability of this is precisely~$y$.
Given $m_{i_{k-1}}$ and $n_{i_k}$, by Lemmas~\ref{lemleftboxtaller}
and~\ref{leftheightsame}, we know precisely that $m_{i_k}=\HEIGHT
_{n_{i_k}} > \HEIGHT_{n_{i_{k-1}}} = m_{i_{k-1}}$.
In other words, we know precisely that the random walk
$L^{n_{i_k},m_{i_{k-1}}}$ has no chance. Thus, for $0 < m < y < 1$,
\begin{eqnarray*}
&&\mathbf{P}\{ m_{i_k} \leq y \mid m_{i_{k-1}}=m \}\\
&&\qquad{}= \mathbf{P}\{
\HEIGHT_{n_{i_k}} \leq y \mid m_{i_{k-1}}=m \} \\
&&\qquad{}=\mathbf{P}\{ L^{n_{i_k},y} \mbox{ has a chance} \mid L^{n_{i_k},m}
\mbox{ has no chance} \} \\
&&\qquad{}= \frac{1 - {\mathbf{P}}\{L^{n_{i_k},y} \mbox{ has no chance}\} -
{\mathbf{P}}\{L^{n_{i_k},m} \mbox{ has a chance}\}}{{\mathbf{P}}\{
L^{n_{i_k},m} \mbox{ has no chance}\}} \\
&&\qquad{}= \frac{y-m}{1-m},
\end{eqnarray*}
which proves the lemma.
\end{pf}

Note also, by the first remark in the proof of the lemma, we have the
following proposition.
\begin{prop}[(\cite{mcdiarmid97mst}, Theorem 2)]
$\HEIGHT_0 \eqdist\operatorname{Uniform}[0,1]$.
\end{prop}

We now prove a more substantial result, about the structure of the
portion of $\IPC(\pff,\Z)$ that
lives ``under the backward maximum process.'' It essentially states
that, like the forward maximal process, the portion of $\IPC(\pff,\Z)$
that lives
under the backward maximum process looks like a single infinite
backbone, to which subcritical Poisson Galton--Watson trees are\vadjust{\goodbreak}
attached at each point.
Also, these subcritical trees become closer and closer to critical the
further along the backbone from $0$ they are.
\begin{thmm}\label{ipcminus} Let $\IPC^{-}(\pff,\Z)$ denote the
restriction of $\IPC(\pff,\Z)$ to the nonpositive integers, and let $0$
be its root. Then $\IPC^{-}(\pff,\Z)$ is distributed as~$\mathcal
{T}_{\mathrm{IIC}}^{*}$.
\end{thmm}

We will prove this theorem at the end of the section.
For each $k \geq0$, let $\pff_k = \pff\cap
([n_{i_{k+1}},n_{i_k}]\times
[0,m_{i_k}))$ and let $\mathcal I_k = \{
n_{i_{k+1}},n_{i_{k+1}}+1,\ldots
,n_{i_k}\}$.
By Lemma~\ref{ipcpgw}, given $(n_{i_{k+1}}-n_{i_k})$, $\IPC(\pff
_k,\mathcal I_k)$ is distributed as $\pgw(m_{i_k})$ conditioned to have
$(n_{i_{k+1}}-n_{i_k})$ nodes, together
with a single additional node (the node $n_{i_k}$) attached to a
uniform vertex.
\begin{thmm}\label{underbackward}
For all $k \geq0$ and $0 < m < 1$, given that $m_{i_k}=m$,\break $\IPC(\pff
_k,\mathcal I_k)$ is distributed as a path with $(1+\operatorname
{Geometric}(m))$ edges, from $n_{i_{k+1}}$ to $n_{i_k}$,
with an independent $\pgw(m)$ tree attached to each node of the path
except $n_{i_k}$.
\end{thmm}
\begin{pf}
The proof uses a correspondence between $\IPC(\pff_k,\mathcal I_k)$ and
a~pond of $\IPC(\pff,\N)$ of appropriate height.
Let $\lambda> 1$ be such that $me^{-m} = \lambda e^{-\lambda}$, so
then $m = \lambda(1-\theta(\lambda))$. For all $n \geq1$, by Lemma
\ref
{heightlengthcond} we have
\begin{eqnarray*}
\mathbf{P}\{ n_{i_{k+1}}-n_{i_k}=n \mid m_{i_k}=m \} & =& (1-m)\cdot
\frac
{(mn)^{n-1} e^{-mn}}{(n-1)!}\\
& =& \frac{n^{n-1}}{(n-1)!} (me^{-m})^n \frac{1-m}{m} \\
& = &\frac{n^{n-1}}{(n-1)!} (\lambda e^{-\lambda})^n \frac{1-\lambda
(1-\theta(\lambda))}{\lambda(1-\theta(\lambda))} \\
& =& \frac{n^{n-1}}{(n-1)!} (\lambda e^{-\lambda})^n \frac{\theta
(\lambda)}{\lambda\theta'(\lambda)} \\
& =& \frac{\theta(\lambda)}{\theta'(\lambda)}\frac{(\lambda n)^{n-1}
e^{-\lambda n}}{(n-1)!}.
\end{eqnarray*}
By Corollary~\ref{forwardsizes}, the latter is the probability that a
pond of $\IPC(\mathcal P,\N)$, conditioned to have height $\lambda$,
has size $n$.
Thus, $\IPC(\pff_k,\mathcal I_k)$ is distributed as a pond of $\IPC
(\mathcal P,\N)$ conditioned to have height $\lambda$. By Theorem
\ref{geom},
it follows that the length of the path from $n_{i_{k+1}}$ to $n_{i_k}$
has distribution $1+\operatorname{Geometric}(\lambda(1-\theta(\lambda)))
\eqdist1+\operatorname{Geometric}(m)$. Furthermore,
by Theorem~\ref{offbackbone}, to each vertex of the path except
$n_{i_k}$ is attached an independent copy of $\pgw(m)$. This completes
the proof.
\end{pf}

Having proved Theorem~\ref{underbackward}, we are now prepared for the
last ingredient needed for the proofs of Theorems~\ref{main1} and \ref
{ipcminus}.\vadjust{\goodbreak}
\begin{thmm}\label{iic}
$\IPC(\pff,\Z)$ is distributed as the Poisson IIC, in the local weak sense.
\end{thmm}
\begin{pf*}{Proof of Theorem \protect\ref{iic}}
For each $j \geq0$, let $T_j$ be the subtree of $\IPC(\pff,\Z)$ rooted
at $n_j$ and containing all nodes reachable from $n_j$ without passing
through $n_{j-1}$ or $n_{j+1}$
(so $T_j$ contains neither $n_{j-1}$ nor $n_{j+1}$).
We show that independently for each $j$, $T_j$ is distributed as $\pgw
(1)$, which proves the theorem.

For each $j$, let $k=k(j)$ be the largest integer $k$ for which $i_k
\leq j$. Let $U_j$ be the
subtree of $T_j$ containing only nodes of index less than $n_{k}$ (in
other words, $U_j$ is the subtree of $T_j$ which lives under the
backward maximum process). Also,
let $V_j$ be the subtree containing $n_j$ and all nodes of $T_j$ not in~$U_j$.
Then by Theorem~\ref{underbackward}, $U_j$ is distributed as $\pgw
(m_{i_{k(j)}})$, independently of $\{U_{j'}\}_{j' \neq j}$.

Next, for each node $\ell\in U_j$, the number of children of $\ell$ in
$\IPC(\pff,\Z)$ that are not in $U_j$ is precisely the number of points
of $\pff$
in $[\ell-1,\ell)\times(m_{i_{k(j)}},1)$, and therefore has
$\operatorname
{Poisson}(1-m_{i_{k(j)}})$ distribution. Furthermore, all such children
have strictly positive index by the
definition of $U_j$.

Finally, as in Theorem~\ref{pgw1}, let $r_1 = \min( i > 0 \dvtx
\lfloor
x(p_i) \rfloor\in U_j )$. Then $r_1$ has a $\operatorname{Poisson}(1)$
number of children, independently of $U_j$.
More generally, exposing the descendants of $U_j$ in a left-to-right
fashion as in Theorem~\ref{pgw1}, we see that each descendant of $U_j$
with positive index has a
$\operatorname{Poisson}(1)$ number of children, independently of all
the others.

To sum up: $U_j$ is distributed as $\pgw(m_{i_{k(j)}})$; each node of
$U_j$ independently has $\operatorname{Poisson}(1-m_{i_{k(j)}})$
children in
$T_j \setminus U_j$; and
each of these children is the root of a $\pgw(1)$ tree, independently
of each other and of $U_j$. It follows that $T_j$ is distributed as
$\pgw(1)$, as claimed.
\end{pf*}
\begin{pf*}{Proof of Theorem \protect\ref{main1}}
The two graphs $T$ and $T'$ are $\BOXES(\pff,\Z)$ and $\IPC(\pff
,\Z)$.
Part (a) follows from Lemma~\ref{lemboxgraph2}. Part (b) is trivial.
Part (c) follows from the example given in Figure~\ref{example2}.
Part (d) is trivial. Part (e) follows from Theorems~\ref{boxesiic}
and~\ref{iic}.
\end{pf*}
\begin{pf*}{Proof of Theorem \protect\ref{main2}}
For each $n$, $\IPC(\pff,\{-n,-n+1,\ldots\})$, viewed as rooted at $0$,
is distributed as $\IPC(\pff,\N)$, viewed as rooted at $n$. The theorem
then follows from the fact that
$\IPC(\pff,\{-n,-n+1,\ldots\}) \to\IPC(\pff,\Z)$ almost surely.
\end{pf*}
\begin{pf*}{Proof of Theorem \protect\ref{main3}}
Let $q\dvtx[1,\infty) \to(0,1]$ be the unique map satisfying the implicit
equation $\frac{d\theta(q(\lambda))}{d\lambda} = -1$ for all
$\lambda
\in[0,\infty)$. (It is possible to write down a more detailed---though still implicit---formula for $q$, but this is unilluminating
and we omit it.) If $W$ is a random variable satisfying\vspace*{-2pt} ${\mathbf{P}}\{
W \leq x\} =
\theta(x)$ for all $x \geq1$, then $q(W) \eqdist\operatorname
{Uniform}[0,1]$, from which the\vadjust{\goodbreak}
first part of the theorem follows immediately. The second part of
Theorem~\ref{main3} then follows from the first, together with
Theorems~\ref{geom}, and~\ref{offbackbone}.
\end{pf*}
\begin{pf*}{Proof of Theorem \protect\ref{ipcminus}}
Let $U_0,U_1,U_2, \ldots$ be the subtrees of $\IPC(\pff,\Z)$ introduced
in the proof of Theorem~\ref{iic}. $\IPC^{-}(\pff,\Z)$ consists exactly
of an infinite backbone path through vertices $n_0,n_1,n_2,\ldots,$
with $U_{j}$ attached at~$n_j$ for each \mbox{$j\ge0$}. Each $U_j$ is
distributed as $\PGW(m_{i_{k(j)}})$. So all that remains to prove the
theorem is to prove that the sequence $(m_{i_{k(j)}})$ is distributed
as the sequence $(M_{j})_{j\ge0}$ defined in the \hyperref[secintro]{Introduction}.
This follows from Theorem~\ref{underbackward} [which states that the
backward maximum process weights $m_{i_{k(j)}}$ stay constant for one
plus a geometric number of values of $j$, with parameter dependent on
the current weight $m_{i_{k(j)}}$, and
the fact that the sequence $(m_{i_{k}})_{k\ge0}$ is as described in
Lemma~\ref{uniformdist}].
\end{pf*}
%

% imsref loaded by akundreckaite, 2011-07-18 15:50:17
%
% imsref loaded by akundreckaite, 2011-07-18 15:56:17
%

%suskaldyti doi

\printaddresses

\end{document}